\documentclass[11pt]{article}

\usepackage{amstext,amssymb,amsmath,amsbsy}
\usepackage{hyperref}
\usepackage{amscd}
\usepackage{amsfonts}
\usepackage{indentfirst}
\usepackage{verbatim}
\usepackage{amsmath}
\usepackage{amsthm}
\usepackage{enumerate}
\usepackage{graphicx}
\usepackage{color}
\usepackage[OT1]{fontenc}
\usepackage[latin1]{inputenc}
\usepackage[english]{babel}
\usepackage{amssymb}
\usepackage{subfig}

\newcommand{\R}{\mathbb{R}}

\newcommand{\J}{\mathcal{J}}
\newcommand{\calL}{\mathcal{L}}
\newcommand{\ds}{\displaystyle}
\newcommand{\Id}{\textrm{Id}}

\newcommand{\x}{{\bf x}}

\newcommand{\F}{{\bf F}}

\newtheorem{Theorem}{Theorem}[section]
\newtheorem{Lemma}{Lemma}[section]

\newtheorem{Proposition}{Proposition}[section]

\newtheorem{remark}{Remark}[section]

\newtheorem*{Assumption*}{Assumption}
\newtheorem{Definition}{Definition}[section]

\newtheorem*{problem*}{Problem}
\numberwithin{equation}{section}

\begin{document}

\title{An inverse source problem for hyperbolic equations and the Lipschitz-like convergence of the quasi-reversibility method}

\author{Loc Hoang Nguyen\thanks{
Department of Mathematics and Statistics, 
	University of North Carolina at Charlotte, 
	Charlotte, NC, 28223, 
	{\it lnguye50@uncc.edu}
 }}

\date{}
\maketitle

\begin{abstract}
	We propose in this paper a new  numerical method to solve  an inverse source problem for general hyperbolic equations. 
	This is the problem of reconstructing sources from the lateral  Cauchy data of the wave field on the boundary of a domain. 
	In order to achieve the goal, we derive an equation involving a Volterra integral, whose
solution directly provides the desired solution of the inverse source problem. 
	Due to the presence of such a Volterra integral, this equation is not in a standard form of partial differential equations. 
		We employ the quasi-reversibility method to find its regularized solution. Using Carleman estimates, we show that the obtained regularized solution converges to the exact solution with the Lipschitz-like convergence rate as the measurement noise tends to $0$.   
		This is one of the novelties of this paper since currently, convergence results for the quasi-reversibility method  are only valid for purely differential equations.	
		Numerical tests demonstrate a good reconstruction accuracy.
\end{abstract}



\section{Introduction}

In this paper, we propose a rigorous numerical method to solve
an inverse source problem (ISP) for a general hyperbolic equation. 
We demonstrate that a modified idea of the Bukhgeim-Klibanov method \cite{BukhgeimKlibanov:smd1981} works for the numerical solution of our ISP. 
The majority of the known methods to solve ISPs are based on the optimization
approach. However, the theory of this approach does not guarantee
convergence of regularized solutions to the exact one when the level of the
noise in the data tends to zero. 
Motivated by this limitation, we establish here the Lipschitz-like convergence of regularized
solutions of our ISP to the true ones. 
We verify our theory numerically.

The inverse source problem (ISP) is the problem of determining a source term from external information about solutions of the governing equations.
The ISP has uncountable applications. 
The important fact is that the desired solutions can be used to directly detect the source even when the source is inactive after a certain time. 
We name here some examples about the applications of ISPs. 
In the case that the governing equation is hyperbolic, the ISP addresses ultrasonics imaging, photoacoustic tomography, seismic imaging \cite{Anastasioelal:ip2007, DeHoopTittelfitz:ip2015, Devaney:itsu1983, JiangLiuYamamoto:jde2017, MulkanovaRomanov:ejmca2016, Yamamoto:ip1995}.
 In the case of the parabolic equation, the ISP plays an important role in various applications \cite{Malyshev:jmaa1989}; for e.g., in identifying the pollution sources of a river or a lake, \cite{AndrleBadia:ipse2015, BadiaDuong:ip2000, BadiaDuong:jiip2002, RapElal:IJNHFF2006}, and in the case of elliptic equation, the ISP arises from electroencephalography, biomedical imaging \cite{BadiaDuong:ip2000, BaoLinTriki:cm2011, ChengIsakovLu:jde2016}.
 Due to its real world applications, the ISP was studied intensively. 
   We refer the reader to  \cite{BaoLinTriki:CRM2011, HusseinLesnic:ejbe2014, HusseinLesnic:jem2016, HusseinLesnic:jem2016t, LesnicHusseinJohabsson:jcam2016, 
ImanuvilovIsakovYamamoto:cpam2003, Isakov:bookSpringer2006} and the references therein for more discussion and mathematical results about various versions of ISPs.

The ISP in this paper is the linearization of severely ill-posed and highly
nonlinear coefficient inverse problems for hyperbolic equations; see e.g., 
\cite{LiuTriggiani:SIAMJMatt2011, LiuTriggiani:bookchapter,
LiuTriggiani:dcds2013}. In these references, the questions about
uniqueness and stability of coefficient inverse problems are addressed via
addressing similar questions for ISPs. We briefly discuss the linearization
issue in Section \ref{sec problem statement}. Hence, besides the direct
application of finding the restoring force, the method of this paper has
potential contributions in sonar imaging, geographical exploration, medical
imaging, near-field optical microscopy, nano-optics, etc. 


In their celebrated paper, Bukhgeim and Klibanov \cite%
{BukhgeimKlibanov:smd1981} in 1981 have paved the way, for the first time,
to prove uniqueness theorems for a large class of inverse problems,
including the ISP of this paper. Indeed, in \cite{BukhgeimKlibanov:smd1981} the powerful tool of Carleman
estimates was introduced for the first time in the field of inverse
problems. Since then, many important related uniqueness results were proved
on the basis of the Bukhgeim-Klibanov method. In this regard, we refer to,
e.g. \cite{ImanuvilovYamamoto:ip2001,
ImanuvilovIsakovYamamoto:cpam2003, Isakov:bookSpringer2006, Klibanov:ip1992,
KlibanovTimonov:u2004}. In addition, the book \cite{BellassouedYamamoto:SpKK2017} extends the
Bukhgeim-Klibanov method to the case of inverse problems for some systems of
PDEs. 
Some extensions of the Bukhgeim-Klibanov method allow one not only to
prove uniqueness theorems but also to establish Hölder and Lipschitz
stability estimates for inverse problems.
It is worth mentioning that recent modifications of  the Bukhgeim-Klibanov method are used to find numerical solutions
of nonlinear coefficient inverse problems via the so-called
\textquotedblleft convexification" method, see, e.g. \cite{KlibanovKolesov:cma2018, KlibanovAlex:SIAMjam:2017}. 
 Surveys of this method can be
found in \cite{BeilinaKlibanovBook, Klibanov:anm2015}. Many stability
results for ISPs were proved this way. 
We list here several important works. 
When the source is in the form of the separation of variables the Hölder
stability result was obtained in \cite{Yamamoto:ip1995}. 
In \cite{ImanuvilovYamamoto:ip2001,
ImanuvilovYamamoto:cpde2001,JiangLiuYamamoto:jde2017,LiuTriggiani:SIAMJMatt2011, LiuTriggiani:bookchapter, LiuTriggiani:dcds2013}
some Lipschitz stability estimates for ISPs for hyperbolic PDEs were
obtained.

For the numerical solutions of ISPs, the widely used approach is
optimization. We draw the reader's attention to several important publications \cite{BourgeoisRecoquillay:esaim2018, 
HusseinLesnic:ejbe2014, HusseinLesnic:jem2016, HusseinLesnic:jem2016t,
JiangLiuYamamoto:jde2017, LiuTriggiani:dcds2013, MulkanovaRomanov:ejmca2016}%
, in some of which good numerical results using optimal control were
obtained. In particular, in \cite{JiangLiuYamamoto:jde2017} 
the ISP, which is similar with the one of the current paper, is considered, a
numerical method is proposed and implemented. The numerical method of \cite%
{JiangLiuYamamoto:jde2017} is based on the optimization approach. The
convergence of regularized solutions to the exact one is not proved in \cite%
{JiangLiuYamamoto:jde2017}. 
To contribute to the field, we propose in this paper a  numerical method, which is not difficult to implement, without using the straight forward optimal control approach. 
First, we derive an integro-differential equation involving a Volterra integral together with lateral Cauchy data. 
Next, we apply the quasi-reversibility method to solve that Cauchy problem numerically.

The quasi-reversibility method was first proposed by Lattès and Lions \cite{LattesLions:e1969}
in 1969. Since then it has been studied intensively \cite%
{Becacheelal:AIMS2015, Bourgeois:ip2005, Bourgeois:ip2006,
BourgeoisDarde:ip2010, ClasonKlibanov:sjsc2007, Dadre:ipi2016,
KlibanovKuzhuget:aa2008, Klibanov:jiipp2013}. The application of Carleman
estimates for proofs of convergence of those minimizers was first proposed
in \cite{KlibanovSantosa:SIAMJAM1991} for Laplace's equation. In particular, \cite{KlibanovMalinsky:ip1991} is the first publication where it was proposed to use Carleman estimates to obtain Lipschitz stability of solutions of hyperbolic equations with lateral Cauchy data.
We draw the reader's attention to
the paper \cite{Klibanov:anm2015} that represents a survey of the quasi-reversibility method.
Using a Carleman estimate, we prove Lipschitz-like convergence rate of regularized solutions generated by the quasi-reversibility method to the exact solution of that
Cauchy problem. The convergence of regularized solutions is known for quasi-reversibility method
for partial differential equations without integrals \cite{Klibanov:anm2015}. The current
publication is the first one where this convergence is proven for the case
of an integro-differential equation with a Volterra integral in it. This can be considered as an important contribution of this paper to the quasi-reversibility method.

The paper is organized as follows. 
In the next section, we state the inverse source problem. 
In Section \ref{sec integro differential equation}, we derive an equation involving a Volterra integral leading to our reconstruction method. In Section \ref{sec QR}, we discuss about the quasi-reversibility method. Section \ref{sec regu} is to prove the Lipschitz stability result for the quasi-reversibility method. 
Then, in Section \ref{sec num}, we present the implementation and numerical examples. Section \ref{sec rem} is for concluding remarks.

\section{The problem statement}\label{sec problem statement}

The ISP we solve in this paper is stated as follows.
Let $d \geq 2$ be the spatial dimension. Let $a \in C^{[(d + 1)/2] + 3}(\R^d)$, $B \in C^{[(d + 1)/2] + 3}(\R^d, \R^d)$ and $c \in C^{[(d + 1)/2] + 3}(\R^d)$ be known coefficients where $[s]$ is the integer part of $s$ for all $s \in \R$. Assume that all functions $c(\x) - 1, a$ and $B$ are compactly supported.

Consider the following problem
\begin{equation}
	\left\{
		\begin{array}{rcll}
			c(\x) u_{tt} - \Delta u &=& a(\x)u +  B(\x) \cdot \nabla u + p(\x) h(\x, t) &(\x, t) \in \R^d \times (0, \infty),\\
			u(\x, 0) &=& f(\x) &\x \in \R^d,\\
			u_t(\x, 0) &=& g(\x) &\x \in \R^d
		\end{array}
	\right.
	\label{hyper eqn}
\end{equation}
where $f$ and $g$ are functions in $H^{[(d + 1)/2 + 5]}(\R^d)$ whose supports are compact. The function $h: \R^d \times [0, \infty)$ is such that $h(\cdot, t) \in C^{[(d + 1)/2] + 1}(\R^d)$ for all $t > 0$ and $h(\x, \cdot) \in C^3([0, \infty)$ for all $\x \in \R^d.$
According to \cite{ImanuvilovYamamoto:ip2001, ImanuvilovYamamoto:cpde2001}, the function $p(\x)$ corresponds to a restoring force. The conditions imposed on $c, a, B, p$ and $h$ above guarantee that the solution $u$ of \eqref{hyper eqn} exists uniquely and belongs to $C^4(\R^d)$, see \cite[Proposition 4.1]{Ladyzhenskaya:sv1985} and \cite[Theorem 2.2]{Klibanov:ip2013}. These technical conditions, however, are not concerned in numerical studies.

Let $T > 0$ be large enough and let $\Omega$ be a piecewise smooth bounded domain of $\R^d$.
We are interested in the problem below. 
\begin{problem*}[Inverse Source Problem (ISP)]
	Fix $T > 0$ and let $u$ be the solution of \eqref{hyper eqn}.
	Assume that $|h(\x, 0)| > 0$ for all $\x \in \Omega.$
	 Determine the function $p(\x)$, $\x \in \Omega$, from the boundary measurements of  the following lateral Cauchy data
	 \begin{equation}
	 	F(\x, t) = u(\x, t) \mbox{ and } G(\x, t) = \partial_{\nu} u(\x, t) \quad (\x, t) \in \partial \Omega \times [0, T].
		\label{data}
	 \end{equation}
\end{problem*}

As mentioned in the Introduction section, the uniqueness of this ISP was first proved by Bukhgeim and Klibanov in [13], also see \cite{BeilinaKlibanovBook, BellassouedYamamoto:Hindaw2006, ImanuvilovYamamoto:ip2001, KlibanovTimonov:u2004, Klibanov:jiipp2013, LiuTriggiani:SIAMJMatt2011, LiuTriggiani:bookchapter}.
Furthermore, the uniqueness for the ISP was established when the data $F(\x, t)$ and $G(\x, t)$ are measured only on a part of $\partial \Omega$ in \cite{BellassouedYamamoto:Hindaw2006} and in a subsomain of $\Omega$ in \cite{ImanuvilovYamamoto:ip2001}.

We would like to roughly show that ISP above serves as an important step in solving coefficient inverse problems, which are well-known to be severely ill-posed and highly nonlinear. 
Assume that we want to reconstruct the function $\mathfrak{c}(\x) \geq 1$ from the measurement of $\mathfrak u$ and $\partial_{\nu} \mathfrak u$ on $\partial \Omega \times [0, T]$ where $\mathfrak u$ is the solution to the following problem
\begin{equation}
	\left\{
		\begin{array}{rcll}
			\mathfrak{c} \mathfrak u_{tt}(\x, t) &=& \Delta \mathfrak u(\x, t) &(\x, t) \in\R^d \times [0, T],\\
			\mathfrak u(\x, 0) &=& \mathfrak f(\x) &\x \in \R^d,\\
			\mathfrak u_t(\x, 0) &=& \mathfrak g(\x) &\x \in \R^d
		\end{array}
	\right.
\label{nonlinear}
\end{equation}
where $\mathfrak f(\x)$ and $\mathfrak g(\x)$ are known as the initial value and velocity, respectively, of the wave.
The function $\mathfrak f(\x)$ is supposed to satisfy the condition $\Delta \mathfrak f$ is non-zero everywhere in the closure of a domain $\Omega$. 
Assume that an initial guess for $\mathfrak c$ is available and denoted by $\mathfrak c_0$.
Let $\mathfrak u_0$ denote the solution of \eqref{nonlinear} when $\mathfrak c(\x)$ is replaced by the function $\mathfrak c_0(\x)$. 
Write 
\[
	\mathfrak v(\x, t) = \mathfrak u(\x, t) - \mathfrak u_0(\x, t).
\]
It is not hard to verify that
\[
	\left\{
		\begin{array}{rcll}
			\mathfrak c_0(\x) \mathfrak v_{tt}(\x, t) &=& \Delta \mathfrak v(\x, t) + (\mathfrak c_0(\x) - \mathfrak \mathfrak c(\x))\mathfrak u_{tt}(\x, t) &(\x, t) \in \R^d \times [0, T],\\
			\mathfrak v(\x, 0) &=& 0 &\x \in \R^d,\\
			\mathfrak v_t(\x, 0) &=& 0 &\x \in \R^d.
		\end{array}
	\right.
\]
Since $\mathfrak c$ is a small perturbation of $\mathfrak c_0$, roughly speaking, we can replace the function $\mathfrak u_{tt}(\x, t)$ in the problem above by $(\mathfrak u_0)_{tt}(\x, t)$ to obtain
\begin{equation*}
	\left\{
		\begin{array}{rcll}
			\mathfrak c_0(\x) \mathfrak v_{tt}(\x, t) &=& \Delta \mathfrak v(\x, t) + \mathfrak p(\x) \mathfrak h(\x, t) &(\x, t) \in \R^d \times [0, T],\\
			\mathfrak v(\x, 0) &=& 0 &\x \in \R^d,\\
			\mathfrak v_t(\x, 0) &=& 0 &\x \in \R^d
		\end{array}
	\right.
\end{equation*}
where $\mathfrak p(\x) = \mathfrak c_0(\x) - \mathfrak c(\x)$ and $\mathfrak h(\x, t) = (\mathfrak u_0)_{tt}(\x, t)$. Note that 
\[
	\mathfrak h(\x, 0) = (\mathfrak u_0)_{tt}(\x, 0) = \Delta \mathfrak u_0(\x, 0) = \Delta \mathfrak f(\x),
\]
which is non-zero everywhere in $\overline \Omega$ by our assumption.
The boundary values of $\mathfrak v(\x, t)|_{\partial \Omega \times [0, T]}$ and $\partial_{\nu} \mathfrak v(\x, t)|_{\partial \Omega \times [0, T]}$ are computed from the knowledge of $\mathfrak u_0$ and $\mathfrak u$.
 The problem of finding $\mathfrak p(\x) = \mathfrak c_0(\x) - \mathfrak c(x)$ to obtain $\mathfrak c(\x)$ is a particular case of our ISP. 
This is, actually, the linearization approach to enhance the accuracy for the solution of coefficient inverse problem.

In the next section, we will derive an integro-differential equation that leads to our numerical method.

\section{A Volterra integro-differential equation} \label{sec integro differential equation}

This section aims to establish an equation whose solution directly yields the solution of the ISP. 
Define
\begin{equation}
	v(\x, t) = u_t(\x, t) \quad (\x, t) \in \Omega \times [0, T].
	\label{2.1}
\end{equation}
It follows from \eqref{hyper eqn} that
\begin{equation}
	 c(\x) v_{tt} - \Delta v = a(\x) v +  B(\x) \cdot \nabla v  + p(\x) h_t(\x, t) \quad (\x, t) \in \Omega \times [0, T].
	 \label{v eqn}
\end{equation}
The function $v$ satisfies the initial conditions
\begin{align}
	&v(\x, 0) = u_t(\x, 0) = g(\x), 
	\label{2.3}
	\\
	&c(\x) v_t(\x, 0) = c(\x)u_{tt}(\x, 0) = \Delta f(\x) + a(\x) f(\x) + B(\x) \cdot \nabla f(\x)  + p(\x) h(\x, 0)
	\label{2.4}
\end{align}
for $\x$ in $\Omega.$
Let $\tilde h: \Omega \times [0, \infty)$ be a smooth function satisfying 
\begin{equation}
	\tilde h(\x, 0) = h(\x, 0), \quad \tilde h_t(\x, 0) = h_t(\x, 0),  \quad  \tilde h(\x, t) \not = 0 
	\label{2.55}
\end{equation}
for all $\x \in \Omega$ and $t \in [0, T].$ An example for such a function is 
\[
	\tilde h(\x, t) = h(\x, 0) \exp(t h_t(\x, 0)/ h(\x, 0)).
\]
Define the function 
\begin{equation}
	w(\x, t) = \frac{v_t(\x, t)}{\tilde h(\x, t)}, \mbox{ so that }  v_t(\x, t) = \tilde h(\x, t) w(\x, t)\quad (\x, t) \in \Omega \times [0, T].
	\label{2.5}
\end{equation}
A straight forward calculation yields 
\begin{multline*}
	c(\x) w_{tt} - \Delta w 
	= 
	 \frac{(c(\x)v_{tt} - \Delta v)_t}{\tilde h(\x, t)}  
	+ c(\x)  \left(\frac{1}{\tilde h(\x, t)}\right)_{tt} v_t
	\\
	+ 2 c(\x)  \left(\frac{1}{\tilde h(\x, t)}\right)_t v_{tt}
	 - \Delta \left(\frac{1}{\tilde h(\x, t)}\right)v_t
	 - 2\nabla \left(\frac{1}{\tilde h(\x, t)}\right) \nabla v_t.
	\end{multline*}
Therefore, by \eqref{v eqn} and \eqref{2.5}, we have
\begin{multline}
	c(\x) w_{tt} - \Delta w 
	=  \frac{a(\x) \tilde h(\x, t) w + B(\x)\cdot \nabla (\tilde h(\x, t) w) + p(\x) h_{tt}(\x, t)} {\tilde h(\x, t)}
	 + c(\x) \left(\frac{1}{\tilde h(\x, t)}\right)_{tt} \tilde h(\x, t) w
	\\
	 + 2 c(\x) \left(\frac{1}{\tilde h(\x, t)}\right)_t (\tilde h(\x, t) w)_t 
	 - \Delta \left(\frac{1}{\tilde h(\x, t)}\right) \tilde h(\x, t) w
	 - 2\nabla \left(\frac{1}{\tilde h(\x, t)}\right) \nabla (\tilde h(\x, t) w ). 
	\label{2.6}
\end{multline}
At the time $t = 0,$  by \eqref{v eqn} and \eqref{2.4},
\begin{align*}
	c(\x)w(\x, 0) &= \frac{c(\x) v_t(\x, 0)}{\tilde h(\x, 0)} \nonumber
	\\
	&= \frac{\Delta f(\x) + a(\x) f(\x) + B(\x) \cdot \nabla f(\x) + p(\x) h(\x, 0)}{\tilde h(\x, 0)}.
\end{align*}
It follows from \eqref{2.55} that
\begin{equation}
	p(\x) = c(\x) w(\x, 0) - \frac{ \Delta f(\x) + a(\x) f(\x) + B(\x) \cdot \nabla f(\x)}{h(\x, 0)}. 
		\label{2.7}
\end{equation}
Since
\[
	w(\x, 0) = w(\x, t) - \int_0^t w_t(\x, \tau) d\tau,
\]
we can rewrite \eqref{2.7} as
\begin{equation*}
	p(\x) = c(\x) w(\x, t) - c(\x) \int_0^t w_t(\x, \tau) d\tau - \frac{ \Delta f(\x) + a(\x) f(\x) + B(\x) \cdot \nabla f(\x)}{h(\x, 0)}.
\end{equation*}
Plugging this into \eqref{2.6}, we derive an equation for $w$ 
\begin{equation}
	c w_{tt} - \Delta w = L(w, \nabla w, w_t) - \frac{c(\x)  h_{tt}(\x, t)}{\tilde h(\x, t)} \int_0^t w_t(\x, \tau) d\tau + \mathcal F(\x, t)
\end{equation}
where
\begin{align*}
	L(w, \nabla w, w_t) &=   \frac{a(\x) \tilde h(\x, t) w + B(\x)\cdot \nabla (\tilde h(\x, t) w) } {\tilde h(\x, t)}
	 + c(\x) \left(\left(\frac{1}{\tilde h(\x, t)}\right)_{tt} \tilde h(\x, t) + \frac{h_{tt}(\x, t)}{\tilde h(\x, t)}\right) w 
	 \nonumber\\
	&\quad
	 + 2 c(\x) \left(\frac{1}{\tilde h(\x, t)}\right)_t (\tilde h(\x, t) w)_t 
	 - \Delta \left(\frac{1}{\tilde h(\x, t)}\right) \tilde h(\x, t) w
	 - 2\nabla \left(\frac{1}{\tilde h(\x, t)}\right) \nabla (\tilde h(\x, t) w  ) 
\end{align*}
and
\begin{equation}
	\mathcal F(\x, t) = - \frac{h_{tt}(\x, t)( \Delta f(\x) + a(\x) f(\x) + B(\x) \nabla f(\x)}{h(\x, 0) \tilde h(\x, t)}.
	\label{2.12}
\end{equation}

We are now in the position to derive some boundary and initial constraints for the function $w$. 
It follows from \eqref{v eqn}, \eqref{2.3}, \eqref{2.4} and \eqref{2.5}, for all $\x \in \Omega$,
\begin{align*}
	w_t(\x, 0) &= \frac{v_{tt}(\x, 0)}{\tilde h(\x, 0)} - \frac{v_t(\x, 0) \tilde h_t(\x, 0)}{\tilde h^2(\x, 0)}
	\\
	&= \frac{\Delta g(\x) + a(\x) g(\x) + B(\x) \cdot \nabla g(\x)  + p(\x) h_t(\x, 0)}{c(\x)\tilde h(\x, 0)} 
	\\ 
	& \quad -\frac{[\Delta f(\x) + a(\x) f(\x) +B(\x) \cdot \nabla f(\x) + p(\x) h(\x, 0)] \tilde h_t(\x, 0)}{c(\x)\tilde h^2(\x, 0)},
\end{align*} 
which is simplified as
\begin{equation*}
	w_t(\x, 0) = \frac{(\Delta g(\x) + B(\x) \cdot \nabla g(\x)) h(\x, 0) - (\Delta f(\x) + B(\x) \cdot \nabla f(\x))}
	{c(\x) h^2(\x, 0)}.
\end{equation*}
On the other hand, by \eqref{data}, \eqref{2.1} and \eqref{2.5}, we can find the boundary data for the function $w$:
\begin{equation}
	w(\x, t) = \frac{F_{tt}(\x, t)}{\tilde h(\x, t)} \mbox{ and } \partial_{\nu} w(\x, t) =  \frac{G_{tt}(\x, t) \tilde h(\x, t) - \partial_{\nu} \tilde h(\x, t) F_{tt}(\x, t)}{\tilde h^2(\x, t)} \quad (\x, t) \in \partial \Omega \times [0, T].
	\label{boundary conditions}
\end{equation}	

The arguments above are summarized as the following proposition. 
\begin{Proposition}  
Define $\tilde h$ as in \eqref{2.55} and $\mathcal F$ as in \eqref{2.12}. Let $u$ be the solution of the hyperbolic problem \eqref{hyper eqn}.
Then the function 
\begin{equation}
	w(\x, t) = \frac{u_{tt}(\x, t)}{\tilde h(\x, t)} \quad \mbox{for all }(\x, t) \in \Omega \times [0, T]
\end{equation}
 satisfies
 \begin{equation}
	 \calL w =  \mathcal F(\x, t) \quad (\x, t) \in \Omega \times [0, T]
	 \label{problem for w}
 \end{equation}
 where 
\begin{equation}
	\mathcal L \phi = c(\x) \phi_{tt} - \Delta \phi - \ds   L(\phi, \nabla \phi, \phi_t) + \frac{c(\x)  h_{tt}(\x, t)}{\tilde h(\x, t)} \int_0^t \phi_t(\x, \tau) d\tau
	\label{calL}
\end{equation}
for all functions $\phi.$ Furthermore,
\begin{equation}
	w_t(\x, 0) = \Psi(\x) 
	\label{initial vel}
\end{equation}
where 
\begin{equation}
	\Psi(\x) = \frac{(\Delta g(\x) + B(\x) \cdot \nabla g(\x)) h(\x, 0) - (\Delta f(\x) + B(\x) \cdot \nabla f(\x))}
	{c(\x) h^2(\x, 0)}
	\label{2.19}
\end{equation}
for all $\x \in \Omega$ and
\begin{equation}
	w(\x, t) = \zeta(\x, t) \mbox{ and } \partial_{\nu} w(\x, t) = \xi(\x, t) \quad \mbox{for all }  (\x, t) \in \partial \Omega \times [0, T]
	\label{boundary}
\end{equation}
where
\begin{equation}
	\zeta(\x, t) = \frac{F_{tt}(\x, t)}{\tilde h(\x, t)}
	\mbox{ and }
	\xi(\x, t) = \frac{G_{tt}(\x, t) \tilde h(\x, t) - \partial_{\nu} \tilde h(\x, t) F_{tt}(\x, t)}{\tilde h^2(\x, t)}.
	\label{zetaxi}
\end{equation}
\label{prop 2.1}
\end{Proposition}

\begin{remark}
Our method to find the solution of the ISP is based on a numerical method to solve \eqref{problem for w}, \eqref{initial vel} and \eqref{boundary} for a function $w_{\alpha}$. The knowledge of $w_{\alpha}$ directly yields that of $p(\x)$ via \eqref{2.7}. 
	 Involving a Volterra integral, equation \eqref{problem for w} is not a standard partial differential equation. A theoretical method to solve it is not yet available. We solve it numerically by the quasi-reversibility method.
\end{remark}

\begin{remark}
From now on, without loss of generality, we consider the functions $\zeta(\x, t)$ and $\xi(\x, t)$, $(\x, t) \in \partial \Omega \times [0, T]$, as available data. They can be computed directly in terms of the measured data $F(\x, t)$ and $G(\x, t)$ via \eqref{zetaxi}.
\label{rem data}
\end{remark}

Throughout the paper, we consider the case when the given data $F$ and $G$ are noisy.
We explain in Section \ref{sec differentiate} how  to differentiate them.
 We require that those functions are admissible in the following sense.

\begin{Definition}[The set of admissible data] The functions $F$ and $G:$   $\partial \Omega \times [0, T] \to \R$ are said to be admissible if and only if the set
	\begin{multline}
	K = \Bigl\{\phi \in H^3(\Omega \times [0, T]): \phi_t(\x, 0) = \Psi(\x) \mbox{ for all } \x \in \Omega, \mbox{ and }\\
	\phi(\x, t) = \zeta(\x, t) \mbox{ and } \partial_{\nu} \phi(\x, t) = \xi(\x, t) \mbox{ for all } (\x, t) \in \partial\Omega \times [0, T]
	\Big\}
	\label{contraints K}
\end{multline}
is not empty. 
\end{Definition}

\begin{remark}
Let 
\begin{multline}
	H = \Bigl\{\phi \in H^3(\Omega \times [0, T]): \phi_t(\x, 0) = 0 \mbox{ for all } \x \in \Omega, \mbox{ and }\\
	\phi(\x, t) = 0 \mbox{ and } \partial_{\nu} \phi(\x, t) = 0 \mbox{ for all } (\x, t) \in \partial\Omega \times [0, T]
	\Big\}.
\end{multline}
It is not hard to see that $H$ is a subspace of $H^3(\Omega \times [0, T])$ and
\[
	K = \phi + H
\] where $\phi$ is a particular element of $K$. Hence, $K$ has no boundary with respect to the topology of $H^3(\Omega \times [0, T])$.
\label{remark 3.3}
\end{remark}


In the next section, we propose a numerical method to solve the Cauchy  problem \eqref{problem for w}, \eqref{initial vel} and \eqref{boundary} with the presence of a Volterra integral.

\section{The quasi-reversibility method} \label{sec QR}

Throughout this section, we assume $F$ and $G$ are admissible so that $K$ is non-empty. 
We have the proposition.

\begin{Proposition} 
 For each $\alpha > 0$, the functional
\[
	\J_{\alpha} (w) = \left\|\calL w - \mathcal F(\x, t)\right\|_{L^2(\Omega \times [0, T])}^2   + \alpha \|w\|_{H^3(\Omega \times [0, T])}^2 
\]
has a unique minimizer $w_{\alpha}$ in $K$. 
\label{prop quasi}
\end{Proposition}

\begin{proof}
	Fix $\alpha > 0$ 
	and let $\{w_n\}_{n = 1}^{\infty} \subset K$ be such that
	\begin{equation}
		\lim_{n \to \infty}\J_{\alpha} (w_n) = \inf_{w \in K} \J_{\alpha} (w).
		\label{3.3}
	\end{equation} 
	We claim that $\{w_n\}_{n = 1}^{\infty}$ is bounded in $H^3(\Omega \times [0, T])$. In fact, if $\{w_n\}_{n = 1}^{\infty}$ has a unbounded subsequence in $H^3(\Omega \times [0, T])$, still named as $\{w_n\}_{n = 1}^{\infty}$, then by \eqref{3.3}
	\begin{align*}
		\infty &= \limsup_{n \to \infty} \alpha \|w_n\|^2_{H^3(\Omega \times [0, T])} 
		\\
		&\leq \limsup_{n \to \infty} \left(\| \calL w_n - \mathcal F(\x, t)\|_{L^2(\Omega \times [0, T])}^2   
		+ \alpha \|w_n\|_{H^3(\Omega \times [0, T])}^2\right) 
		= \inf_{w \in K} \J_{\alpha}(w),
	\end{align*}
	which is impossible. 
	Since $\{w_n\}_{n = 1}^{\infty}$ is bounded, it has a  subsequence, still called $\{w_n\}_{n = 1}^{\infty}$, weakly converges to a function $w_0$ in $H^3(\Omega \times [0, T])$. Since $K$ is closed and convex in $H^3(\Omega \times [0, T])$, by \cite[Theorem 3.7]{Brezis:Springer2011}, $K$ is weakly closed in $H^3(\Omega  \times [0, T])$. Thus, $w_0$ is in $K$.
	By the compact continuous embedding from $H^3(\Omega \times [0, T])$ to $H^2(\Omega \times [0, T])$, without lost of generality, we can assume that $\{w_n\}_{n = 1}^{\infty}$ converges strongly to $w_0$ in $H^2(\Omega \times [0, T])$. The function $w_0$ is a minimizer of $\J_{\alpha}$. In fact,
\begin{align*}
	 \J_{\alpha}(w_0) 
		&= \| \calL w_0 - \mathcal F(\x, t)\|_{L^2(\Omega \times [0, T])}^2
		+ \alpha \|w_0\|_{H^3(\Omega \times [0, T])}^2 
		\\
		 &\leq \lim_{n \to \infty} \| \calL w_n - \mathcal F(\x, t)\|_{L^2(\Omega \times [0, T])}^2 + \alpha \limsup_{n \to \infty} \|w_n\|_{H^3(\Omega \times [0, T])}^2 
		 \\
		&= \limsup_{n \to \infty} \J_{\alpha}(w_n) = \inf_{w \in K} \J_{w \in K}(w).
	\end{align*}
On the other hand,	
since $\J_{\alpha}$ is strictly convex and $K$ has no boundary (see Remark \ref{remark 3.3}), $\J_{\alpha}$ has only one minimizer. 

\end{proof}

\begin{Definition}[Regularized solution \cite{BeilinaKlibanovBook, Tihkonov:kapg1995}]
	The unique minimizer $w_{\alpha} \in K$ of $\J_{\alpha}$, $\alpha > 0$, is called the regularized solution of problem \eqref{problem for w}, \eqref{initial vel} and \eqref{boundary}.
\end{Definition}

\begin{remark}
	The non-empty condition imposed on $K$ is necessary for the theoretical purpose. However, it is not a serious concern in computation. 
	In fact, we find the minimizer $w_{\alpha}$ of $\J_{\alpha}$  by directly solving the equation $D\J_{\alpha}(w_{\alpha}) = 0$ where $D\J_{\alpha}$ is the Fr\'echet derivative of $\J_{\alpha}$. In the finite difference scheme, this equation and the constraint $w_{\alpha} \in K$ constitute a linear system, say for e.g., $A w_{\alpha} = b$.	
	Since we do not check if $K$ is non-empty, that linear system might not have a solution. We, therefore, approximate $w_{\alpha}$ by the solution of $(A^T A  + \epsilon \Id) w = A^T b$. 
	\label{rem 4.1}
\end{remark}

In the next section, we prove that the regularized solution obtained by the quasi-reversibility method converges to the true solution of \eqref{problem for w}, \eqref{initial vel} and \eqref{boundary} as $\delta$, the noise in measurement, and $\alpha$, the regularized parameter, tend to $0$. The convergence rate is $O(\delta + \sqrt{\alpha}).$

\section{The convergence of the quasi-reversibility} \label{sec regu}

In this section, we study the convergence of the regularized solution to the true solution as the noise level and the regularized parameter tend to $0$.
\subsection{The main result}

Let $R$ be a large positive number such that $\Omega \Subset D = B(R)$.
We impose the following condition on the function $c(\x)$. Assume that there exists a point $\x_0$ in $D \setminus \overline \Omega$ such that
\begin{equation}
	\langle \x - \x_0, \nabla c(\x) \rangle \geq 0 \quad \mbox{for all } \x \in \R^d.
	\label{1.5}
\end{equation}

Recall that data for the ISP are given by the functions $F$ and $G$. 
As mentioned in Remark \ref{rem data}, we can calculate $\zeta$ and $\xi$ via \eqref{zetaxi} in terms of $F$ and $G$. These two functions serve as new data for the ISP.
Let $F^*$ and $G^*$ be the noiseless ``direct" data. Denote by  $\zeta^*$ and $\xi^*$  the functions defined in \eqref{zetaxi} when $F$ and $G$ are replaced by $F^*$ and $G^*$ respectively.
  Since $K$ is non-empty, so is the set 
\begin{multline}
		\mathcal{H} = \Big\{\phi \in H^3(\Omega \times [0, T]): \phi_t(\x, 0) = 0 \mbox{ for all } \x \in \Omega, 
		\\
		\mbox{ and } \phi(\x, t) = \zeta(\x, t) - \zeta^*(\x, t) \mbox{ and } \partial_{\nu} \phi(\x, t) = \xi(\x, t) - \xi^*(\x, t) \mbox{ for all } (\x, t) \in \partial \Omega \times [0, T]\Big\}.
		\label{calH}
\end{multline}
Define the following quantity
\begin{equation}
	\|\zeta - \zeta^*, \xi - \xi^*\|_\mathcal{H} = \inf_{\phi \in \mathcal{H}} \{\|\phi\|_{H^3(\Omega \times [0, T])}\},
	\label{norm of data}
\end{equation}
which is considered as the measured noise.
	By the trace theory, we can verify that the norm $\|\cdot\|_{\mathcal{H}}$ is stronger than the $L^2$ norm in the following sense
	\begin{equation*}
		\|\zeta - \zeta^*, \xi - \xi^*\|_{L^2(\partial \Omega \times [0, T])} \leq C\|\zeta - \zeta^*, \xi - \xi^*\|_\mathcal{H}
	\end{equation*}
	for some constant $C > 0.$
As a consequence, if the measurement noise is small with respect to $\|\cdot\|_{\mathcal{H}}$, then it is small in the $L^2$ sense.


The following theorem is the main result of this paper.


\begin{Theorem}
	Suppose that $K$ is non-empty and that condition \eqref{1.5} holds true. 
		Assume that	\begin{equation}
			\|\zeta - \zeta^*, \xi - \xi^*\|_{\mathcal{H}} \leq \delta
	\label{5.5}
	\end{equation} for some $0 \leq \delta < 1$.
	Fix $\alpha \in (0, 1)$ and let $w_{\alpha}$  be the regularized solution  of \eqref{problem for w}, \eqref{initial vel} and \eqref{boundary}. Let $w^* = \ds \frac{u_{tt}}{\tilde h}$ be the true solution of \eqref{problem for w}, \eqref{initial vel} and \eqref{boundary}. 
	Then, if $T > \left(\max_{\x \in \overline \Omega}{|\x - \x_0|^2}/\eta_0\right)^{1/2}$ where $\eta_0$ is the number, that will be indicated in Lemma \ref{lem 4.1}, the following estimate is true
	\begin{equation}
		\|w_{\alpha} - w^*\|^2_{H^1(\Omega \times [0, T])} \leq C \Big( \delta^2  + \alpha \| w^*\|_{H^3(\Omega \times [0, T])}^2	\Big)
		\label{5.6}
	\end{equation}
	\label{thm convergence}
	for some constant $C = C(D, \x_0, \Omega, T, a, B, c, h).$ As a result,
	\begin{equation}
		\|p_{\delta, \alpha} - p^*\|_{L^2(\Omega)}^2 \leq C \Big( \delta^2  + \alpha \| w^*\|_{H^3(\Omega \times [0, T])}^2	\Big)
		\label{Lipschitz}
	\end{equation}
	where $p_{\delta, \alpha}$ is computed via \eqref{2.7} when $w$ is replaced by $w_{\alpha}$ and $p^*$ is the true source.
	\label{theorem 5.1}
\end{Theorem}


We recall here the Carleman estimates for the reader's convenience. It is important mentioning that these results play a crucial role for the proof of Theorem \ref{thm convergence}. Introduce the function
\begin{equation}
	W(\x, t) = \exp(\lambda(|\x - \x_0|^2 - \eta t^2)), \quad (\x, t) \in \R^d \times [0, \infty)
	\label{CWF}
\end{equation}
where $\lambda$ and $\eta$ are two positive numbers. The function $W(\x, t)$ is known as the Carleman weight function. For  $\eta > 0$ and $\epsilon > 0$, define 
\begin{equation}
	D_{\eta, \epsilon} = \{(\x, t) \in D \times [0, \infty): |\x - \x_0|^2 - \eta t^2 > \epsilon\}.
	\label{2.333}
\end{equation}
We have the lemmas.

\begin{Lemma}  [Lemma 6.1 in \cite{Klibanov:anm2015} and Theorem 1.10.2 in \cite{BeilinaKlibanovBook}]
 Assume condition \eqref{1.5} holds true. 
 Then, there exists a sufficiently small number $\eta_0 = \eta_0\left(\x_0, D, \|c\|_{C^1(\overline D)}\right) \in (0, 1)$ such that for any $\eta \in (0, \eta_0],$ one can choose a sufficiently large number $\lambda_0 = \lambda_0(D, \eta, c, \x_0) > 1$ and a number $C = C(D, \eta, c, \x_0)$  such that for all $z \in C^2(\overline {D_{\eta, \epsilon}})$ and for all $\lambda \geq \lambda_0$, the following pointwise Carleman estimate holds true
\begin{multline}
 	|c(\x) z_{tt}(\x, t) - \Delta z(\x, t)|^2 W^2(\x, t) 
 	\geq C \lambda \Big(|\nabla z(\x, t)|^2 + z_t^2(\x, t) 
  + \lambda^2 z^2(\x, t)\Big) W^2(\x, t)
  \\
 	+ \nabla \cdot Z(\x, t) + Y_t(\x, t)
	\label{5.10}
 \end{multline}
 in $D_{\eta, \epsilon}.$
 The vector valued function $Z(\x, t)$ and the function $Y(\x, t)$ satisfy
 \begin{align}
 	|Z(\x, t)| &\leq C \lambda^3 (|\nabla z(\x, t)|^2 + z_t^2(\x, t) + z^2(\x, t)) W^2(\x, t) \label{4.3}
	\\
	|Y(\x, t)| &\leq C \lambda^3(|t|(|\nabla z(\x, t)|^2 + z_t^2(\x, t) + z^2(\x, t)) + (|\nabla z(\x, t)| + |z|)|z_t(\x, t)| )W^2(\x, t).
	\label{4.4}
 \end{align}
 In particular, if either $z(\x, 0) = 0$ or $z_t(\x, 0) = 0$, then $Y(\x, 0) = 0$.
 \label{lem 4.1}
\end{Lemma}

\begin{Lemma}[Lemma 1.10.3 \cite{BeilinaKlibanovBook}]
	Let the function $\varphi \in C^1[0, a]$ and $\varphi'(t) \leq -b$ in $[0, a]$, where $b = const > 0$. For a function $g \in L^2(-a, a)$, consider the integral
	\[
		I(g, \lambda) = \int_{-a}^a \left(\int_0^t g(\tau) d\tau\right)^2 \exp(2\lambda \varphi(t^2)) dt.
	\]
	Then,
	\[	
		I(g, \lambda) \leq \frac{1}{4\lambda b} \int_{-a}^a g^2(t) \exp(2\lambda \varphi(t^2)) dt.
	\]
	\label{lem integral}
\end{Lemma}

We refer the reader to \cite{BeilinaKlibanovBook, Klibanov:anm2015} for the proof of Lemmas \ref{lem 4.1} and \ref{lem integral}.

\subsection{The proof of Theorem \ref{thm convergence}}

	The arguments in this section follow the ideas of Klibanov in  \cite{BeilinaKlibanovBook, KazemiKlibanov:aa1993, Klibanov:anm2015, KlibanovMalinsky:ip1991}, in which the Lipschitz  stability was established when the operator $\calL$ does not involve the Volterra integral. 
	Note that this stability estimate was first established in \cite{KlibanovMalinsky:ip1991}.
	In this subsection, $C$ denotes a generic constant depending only on known sets and functions: $D, \x_0, \Omega, T, a, B, c, h$.  The number $C$ might change from estimate to estimate.

{\bf Step 1.} 
Let
\begin{multline*}
 H = \left\{\varphi \in H^3(\Omega \times [0, \infty)): \varphi_t(\x, 0) = 0 \mbox{ for } \x \in \Omega \right. 
 \\
\left. \mbox{ and } \varphi(\x, t) = \partial_{\nu}\varphi(\x, t) = 0 \mbox{ for } (\x, t) \in \partial \Omega \times [0, T] \right\}
 \end{multline*}
 be the set of test functions.
 Since the regularized solution $w_{\alpha}$ is the minimizer of $\J_{\alpha}$ on $K$, for all $\varphi \in H,$ we have
\begin{equation}
	  \langle \calL w_{\alpha} - \mathcal F, \calL \varphi\rangle_{L^2(\Omega \times [0, T])} + \alpha \langle w_{\alpha}, \varphi \rangle_{H^3(\Omega \times [0, T])}= 0.
	  \label{4.8}
\end{equation}
 On the other hand,
since $ w^*$ is the true solution of the Volterra integro-differential equation  \eqref{problem for w}, we have
\begin{equation}
	  \langle \calL w^* - \mathcal F, \calL \varphi\rangle_{L^2(\Omega \times [0, T])} + \alpha \langle w^*, \varphi \rangle_{H^3(\Omega \times [0, T])} = \alpha \langle w^*, \varphi \rangle_{H^3(\Omega \times [0, T])}.
	  \label{4.9}
\end{equation}
Subtracting \eqref{4.8} from \eqref{4.9}, we have
\begin{equation}
	\langle \calL (w_{\alpha} - w^*), \calL \varphi\rangle_{L^2(\Omega \times [0, T])} + \alpha\langle w - w^*, \varphi\rangle_{H^3(\Omega \times [0, T])} 
	=
	-\alpha \langle w^*, \varphi \rangle_{H^3(\Omega \times [0, T])}
	\label{4.11}
\end{equation}
for any $\varphi \in H.$
By \eqref{norm of data} and assumption \eqref{5.5}, there is an ``error" function  $\mathcal E$  in $\mathcal{H}$ such that
\begin{equation}
	\|\mathcal E\|_{H^3(\Omega \times [0, T])} \leq C \delta.
	\label{error function}
\end{equation}
Since $w_{\alpha} \in K$, by \eqref{calH}, it is obvious that 
\begin{equation}
	z = w_{\alpha} - w^* - \mathcal E
	\label{z}
\end{equation} is in $H$. Choosing $\varphi = z$ as a test function for \eqref{4.11}, we have
\begin{equation}
	\|\calL z\|_{L^2(\Omega \times [0, T])}^2 + \alpha\| z\|_{H^3(\Omega \times [0, T])}^2 
	= - \langle \calL \mathcal E, \calL z \rangle_{L^2(\Omega \times [0, T])}^2 - \alpha  \langle  \mathcal E,  z \rangle_{H^3(\Omega \times [0, T])}^2 -\alpha \langle w^*, z \rangle_{H^3(\Omega \times [0, T])}.
	\label{5.4}
\end{equation}
On the other hand, using the inequality $ab \leq \frac{a^2}{8} + 2 b^2$ and \eqref{5.4} and noting that $0 < \alpha < 1$, we have
\begin{equation}
	\|\calL z\|_{L^2(\Omega \times [0, T])}^2 + \alpha\| z\|_{H^3(\Omega \times [0, T])}^2 
	\leq C \Big( \delta^2  
	+ \alpha \| w^*\|_{H^3(\Omega \times [0, T])}^2
	\Big).
	\label{step 1}
\end{equation}

\noindent {\bf Step 2.} 
Recall that
\[
	P_1 = \min_{\x \in \Omega}\{|\x - \x_0|\}, \quad P_2 = \max_{\x \in \Omega}\{|\x - \x_0|\}.
\]
Let $\eta_0$ be the number as in Lemma \ref{lem 4.1} and $\epsilon$ be a small positive number. 
Introduce the cut-off function $\chi(\x, t)$ satisfying $\chi_t(\x, t) \leq 0$ and
\begin{equation}
	\chi(\x, t) = \left\{
		\begin{array}{ll}
			1 & \mbox{if } |\x - \x_0|^2 - \eta_0 t^2 > 2\epsilon,\\
			\in (0, 1) & \mbox{if } |\x - \x_0| - \eta_0 t^2 > \epsilon,\\
			0 & \mbox{otherwise}
		\end{array}
	\right. 
	\label{5.15}
\end{equation}
and define the set
\[
	Q_{\epsilon} = \{(\x, t) \in \Omega \times [0, T]: |\x - \x_0|^2 - \eta_0 t^2 > \epsilon\} = D_{\eta_0, \epsilon} \cap \Omega \times [0, T]
\]
where $D_{\eta_0, \epsilon}$ is as in \eqref{2.333}.
We next apply Lemma \ref{lem 4.1}  to get \eqref{5.10} for the function $z$. Multiply $\chi(\x, t)$ to both side of \eqref{5.10} and then integrate the resulting on $Q_{\epsilon}$. We have
\begin{multline}
	\int_{Q_{\epsilon}}W^2(\x, t) \chi(\x, t) (z_{tt} - \Delta z)^2 d\x dt 
	\geq
	 C \lambda \int_{Q_{\epsilon}} \chi(\x, t) \left(|\nabla z|^2 + z_t^2 + \lambda^2 z^2\right) W^2(\x, t) d\x dt
	 \\
 	+ \int_{Q_{\epsilon}} \chi(\x, t) \nabla \cdot Z d\x dt + \int_{Q_{ \epsilon}} \chi(\x, t) Y_t d\x dt
	\label{4.14}
\end{multline}
where $Z$ and $Y$ satisfy \eqref{4.3} and \eqref{4.4} respectively.

Since $z \in H$, it follows from \eqref{4.3} that $Z(\x, t) = 0$ for all $\x \in \partial \Omega$ and $t \in [0, \hat t_{\epsilon}(\x)]$ where
\begin{equation}
	\hat t_{s}(\x) = \sqrt{\frac{|\x - \x_0|^2 - s}{\eta_0}}, \quad s \in (0, P_1^2)
	\label{tx}
\end{equation}
Recall that
\[
	\chi(\x, t) \nabla \cdot Z = \nabla \cdot (\chi(\x, t) Z(\x, t)) - \nabla \chi(\x, t) \cdot Z(\x, t).
\]
 Hence, by \eqref{5.15} 
\begin{align}
	\Big|\int_{Q_{ \epsilon}} \chi(\x, t) \nabla \cdot Z  dt d\x\Big|
	&=  \Big|\int_{\Omega} \int_0^{\hat t_{\epsilon}(\x)} \chi(\x, t) \nabla \cdot Zdt d\x \Big|
	\nonumber\\
	&\leq \Big|\int_{\Omega} \int_0^{\hat t_{\epsilon}(\x)}\nabla \cdot (\chi(\x, t) Z)dt d\x 
	- \int_{\Omega} \int_0^{\hat t_{\epsilon}(\x)} \nabla \chi(\x, t) \cdot Zd t d\x\Big|. \label{5.2121}
\end{align}
Since $z(\x, t) \in H$, both $z(\x, t)$ and $\partial_{\nu} z(\x, t)$ are identically $0$ on $\partial \Omega \times [0, T].$ By \eqref{4.3}, $Z = 0$ on $\partial \Omega \times [0, T].$ Hence, the first integral on the right hand side of \eqref{5.2121} vanishes.
By \eqref{4.3}, we have
\[
	\Big|\int_{Q_{ \epsilon}} \chi(\x, t) \nabla \cdot Z  dt d\x\Big| \leq C\lambda^3 \int_{\Omega} \int_{\hat t_{2\epsilon}(\x)}^{\hat t_{\epsilon}(\x)} W^2(\x, t)(|z|^2 + |\nabla z|^2 + |z_t|^2) dt d\x.
\]
Noting that
\begin{equation}
	W^2(\x, t) \leq \exp(4\lambda \epsilon) \quad \mbox{for all } \x \in \Omega, t \in [\hat t_{2\epsilon}(\x), \hat t_{\epsilon}(\x)],
	\label{5.16}
\end{equation}
we obtain
\begin{equation}
	\Big|\int_{Q_{ \epsilon}} \chi(\x, t) \nabla \cdot Z  dt d\x\Big|
	 \leq C\lambda^3 \exp(4\lambda \epsilon) \int_{\Omega} \int_{\hat t_{2\epsilon}(\x)}^{\hat t_{\epsilon}(\x)} (|z|^2 + |\nabla z|^2 + |z_t|^2) dt d\x.
	 \label{5.17}
\end{equation}
We next estimate the last term in \eqref{4.14}.
Since $z_t(\x, 0) = 0$, it follows from \eqref{4.4} that $Y(\x, 0) = 0.$ 
 We have
\begin{align*}
	\Big|\int_{Q_{\epsilon}} \chi(\x, t) Y_t(\x, t) d\x dt\Big| 
	&= \Big| \int_{\Omega} \int_{0}^{\hat t_{\epsilon}(\x)}  \chi(\x, t) Y_t(\x, t)  dt  d\x \Big|
	\\
	&= \Big|\int_{\Omega}  \int_0^{\hat t_{\epsilon}(\x)} \left[(\chi(\x, t) Y(\x, t))_t - \chi_t(\x, t) Y(\x, t) \right]d\tau  d\x \Big|.
\end{align*}
By \eqref{5.15}, $\nabla \chi(\x, t) = 0$ for all $t \in (0, \hat t_{2\epsilon}(\x))$.
Hence, by \eqref{4.4} and \eqref{5.16},
\begin{equation}
	\Big|\int_{Q_{\epsilon}} \chi(\x, t) Y_t d\x dt\Big| 
	\leq 
	C \lambda^3 \exp(4\lambda \epsilon) \Big| \int_{\Omega} \int_{\hat t_{2\epsilon}(\x)}^{\hat t_{\epsilon}(\x)}  (z^2 + |\nabla z|^2 +z_t^2) d\x dt\Big|.
	\label{5.18}
\end{equation}

It follows from  \eqref{4.14}, \eqref{5.17} and \eqref{5.18} that
\begin{multline}
	\int_{Q_{ \epsilon}}W^2(\x, t) \chi(\x, t) (z_{tt} - \Delta z)^2 d\x dt 
	\geq C \lambda \int_{Q_{\eta_0, \epsilon}} \chi(\x, t) \left(|\nabla z|^2 + z_t^2 + \lambda^2 z^2\right) W^2(\x, t) d\x dt
	\\
	- C \lambda^3 \exp(4\lambda \epsilon) \int_{\Omega} \int_{\hat t_{2\epsilon}(\x)}^{\hat t_{\epsilon}(\x)}  (|\nabla z|^2 + z_t^2 + z^2) d\x dt.
	\label{4.17}
\end{multline}
It follows from \eqref{calL}  that
\begin{multline*}
	\int_{Q_{\epsilon}}  W^2(\x, t) \chi(\x, t) |\calL z|^2 d\x dt  \geq  \int_{Q_{\epsilon}}W^2(\x, t) \chi(\x, t) (z_{tt} - \Delta z)^2 d\x dt 
	\\
	 - C \int_{Q_{\epsilon}}W^2(\x, t) \chi(\x, t) (|\nabla z|^2 + z_t^2 + z^2) d\x dt 
	 - C\int_{Q_{\epsilon}}W^2(\x, t) \chi(\x, t) \Big|\int_0^t z_t(\x, \tau) d\tau\Big|^2 d\x dt.
\end{multline*}
Using this, \eqref{4.17} and the fact that $\chi(\x, t) \leq 1$, we obtain
\begin{multline}
	\int_{Q_{ \epsilon}}  W^2(\x, t) \chi(\x, t) |\calL z|^2 d\x dt  \geq  
	  C (\lambda - 1) \int_{Q_{ \epsilon}}W^2(\x, t) \chi(\x, t) (|\nabla z|^2 + z_t^2 + z^2)  dt d\x 
	  \\
	  - C \lambda^3 \exp(4\lambda \epsilon)\int_{\Omega} \int_{\hat t_{2\epsilon}(\x)}^{\hat t_{\epsilon}(\x)} (|\nabla z|^2 + z_t^2 + z^2)  dt d\x
	 - C\int_{Q_{\epsilon}}W^2(\x, t)   \Big|\int_0^t z_t(\x, \tau) d\tau\Big|^2  dt d\x.
	 \label{5.19}
\end{multline}

\noindent {\bf Step 3.}
We next estimate the Volterra-integral 
\[\ds\int_{Q_{\epsilon}}W^2(\x, t) \left|\int_0^t z_t(\x, \tau) d\tau\right|^2 dt d\x
\] in the right hand side of \eqref{5.19}.
We have 
\begin{multline}
	\int_{Q_{\epsilon}}W^2(\x, t) \left(\int_{0}^t z_t(\x, \tau) d\tau\right)^2 dt d\x
	\\
	= \int_{\Omega} \exp\left(2\lambda |\x - \x_0|^2\right)  
 \int_0^{\hat t_{\epsilon}(\x)} \exp(-2\lambda \eta_0 t^2)\left(\int_{0}^t z_t(\x, \tau) d\tau\right)^2  dt  d\x.
 \label{5.2020}
\end{multline}
Letting $I$ denote the integral with respect to $t$,
extending $z$ as an even function i.e. $z(\x, -t) = z(\x, t)$ and applying Lemma \ref{lem integral} with $\varphi(t) = -\eta_0 t$, we have
\[
	I \leq \frac{1}{4\lambda \eta_0}\int_0^{\hat t_{\epsilon}(\x)} \exp(-2\lambda \eta_0 t^2) \chi(\x, t) |z_t(\x, t)|^2 dt.
\]
By \eqref{5.2020},
\[
	\int_{Q_{\epsilon}}W^2(\x, t) \left(\int_{0}^t z_t(\x, \tau) d\tau\right)^2 d\x dt 
	\leq
	\frac{1}{4\lambda \eta_0} \int_{Q_{\epsilon}}W^2(\x, t) |z_t|^2d\x dt.
\]

Since $\chi(\x, t) = 1$ for $0 < t < \hat t_{2\epsilon}(\x)$,  we have
\[
	\Big|\int_{Q_{\epsilon}} W^2(\x, t) (\chi(\x, t) - 1) |z_t|^2d\x dt \Big| \leq C \lambda^3 \exp(4\lambda \epsilon) \int_{\Omega} \int_{\hat t_{2\epsilon}(\x)}^{\hat t_{\epsilon}(\x)}  (|\nabla z|^2 + z_t^2 + z^2) d\x dt.
\]
Hence, noting that $\lambda \gg 1$, we obtain by \eqref{5.19} that
\begin{multline}
	\int_{Q_{\epsilon}}  W^2(\x, t) \chi(\x, t) |\calL z|^2 d\x dt  \geq  
	  C \lambda  \int_{Q_{ \epsilon}}W^2(\x, t) \chi(\x, t) (|\nabla z|^2 + z_t^2 + z^2) d\x dt 
	  \\
	  - C \lambda^3 \exp(4\lambda \epsilon) \int_{\Omega} \int_{\hat t_{2\epsilon}(\x)}^{\hat t_{\epsilon}(\x)}  (|\nabla z|^2 + z_t^2 + z^2) d\x dt.
	 \label{5.20}
\end{multline}

\noindent {\bf Step 4.} 

In this step, we estimate $\|z\|_{H^1(Q_{2\epsilon})}$.
Note that
\[
	 \chi(\x, t) \leq 1 \mbox{ and } W(\x, t) \leq \exp(2\lambda P_2^2 ) \quad \mbox{for all } (\x, t) \in Q_{\epsilon}.
\]
Using \eqref{step 1} gives
\begin{align*}
	\int_{Q_{ \epsilon}}W^2(\x, t) \chi(\x, t)|\calL z|^2 d\x dt 
	&\leq \exp(2\lambda P_2^2 )\left(\|\calL z\|_{L^2(\Omega \times [0, T])}^2 + \alpha\| z\|_{H^3(\Omega \times [0, T])}^2\right) 
	\\
	&\leq C \exp(2\lambda P_2^2 )\Big( \delta^2  
	+ \alpha \| w^*\|_{H^3(\Omega \times [0, T])}^2
	\Big).
\end{align*}
By \eqref{5.20}
\begin{multline*} 
	   \lambda  \int_{Q_{\epsilon}}W^2(\x, t) \chi(\x, t) (|\nabla z|^2 + z_t^2 + z^2) d\x dt 
	  \leq
	  C \lambda^3 \exp(4\lambda \epsilon) \int_{\Omega} \int_{\hat t_{2\epsilon}(\x)}^{\hat t_{\epsilon}(\x)}  (|\nabla z|^2 + z_t^2 + z^2) d\x dt 
	  \\
	  + C \exp(2\lambda P_2^2 )\Big( \delta^2  
	+ \alpha \| w^*\|_{H^3(\Omega \times [0, T])}^2
	\Big).
\end{multline*}
Since $Q_{4\epsilon} \subset Q_{\epsilon}$ and 
\[
	\chi(\x, t) = 1 \mbox{ and } W^2(\x, t) \geq \exp(8\lambda \epsilon) \quad \mbox{for all } (\x, t) \in Q_{4\epsilon},
\]
we have
\begin{multline} 
	   \lambda  \exp(8\lambda \epsilon) \int_{Q_{4\epsilon}}  (|\nabla z|^2 + z_t^2 + z^2) d\x dt
	   \\ 
	  \leq
	  C \lambda^3 \exp(4\lambda \epsilon) \int_{\Omega} \int_{\hat t_{2\epsilon}(\x)}^{\hat t_{\epsilon}(\x)}  (|\nabla z|^2 + z_t^2 + z^2) d\x dt 
	  + C \exp(\lambda P_2^2 )\Big( \delta^2  
	+ \alpha \| w^*\|_{H^3(\Omega \times [0, T])}^2
	\Big).
	\label{5.21}
\end{multline}
It follows from \eqref{5.21} that
\begin{multline}
	\int_{Q_{4\epsilon}}  (|\nabla z|^2 + z_t^2 + z^2) d\x dt  \leq 
	\frac{\exp(-4\lambda \epsilon)}{\lambda }\int_{\Omega} \int_{\hat t_{2\epsilon}(\x)}^{\hat t_{\epsilon}(\x)}  (|\nabla z|^2 + z_t^2 + z^2) d\x dt 
	  \\
	  + C \exp(2\lambda (P_2^2 - 4\epsilon))\Big( \delta^2  
	+ \alpha \| w^*\|_{H^3(\Omega \times [0, T])}^2 \Big).
	\label{5.25}
\end{multline}


\noindent {\bf Step 5.}

For all $t \in [0, T]$, using \eqref{calL}, we have 
\begin{align*}
	\int_{\Omega} z_t (z_{tt} - \Delta z + z) d\x 
	&\leq 
	C \left(\int_{\Omega} |z_t\calL z| d\x + \int_{\Omega} |z_t|(|z| + |z_t| + |\nabla z| )d\x  + \int_{\Omega} |z_t|\left|\int_{0}^t z_td\tau\right| d\x
	\right)
	\\
	&= C \left(\int_{\Omega} |z_t\calL z| d\x + \int_{\Omega} |z_t|(|z| + |z_t| + |\nabla z| )d\x  + \int_{\Omega} |z_t|\left|z - z(\x, 0)\right| d\x
	\right).
\end{align*}
Using \eqref{step 1} and the inequality $2ab \leq a^2 + b^2$, we have for all $t \in [0, T]$
\begin{multline*}
	\int_{\Omega} z_t (z_{tt} - \Delta z + z) d\x 
	\leq 
	 C\Big(\int_{\Omega} (|z|^2 + |z_t|^2 + |\nabla z|^2)d\x
	+ \int_{\Omega}|z(\x, 0)|^2 d\x
	\\
	+ \int_{\Omega}|\calL z|^2d\x
	+ \delta^2 + \alpha \|w^*\|_{H^3(\Omega \times [0, T])}^2\Big).
\end{multline*}
Integrating by parts for the left hand side of the inequality above, we have
\begin{equation}
	 E'(t)
	\leq 
	 C\left(E(t)
	+ \int_{\Omega}|z(\x, 0)|^2 d\x
	+ \int_{\Omega}|\calL z|^2d\x
	+ \delta^2 + \alpha \|w^*\|_{H^3(\Omega \times [0, T])}^2\right)
	\label{5.26}
\end{equation}
where
\begin{equation}
	E(t) = \int_{\Omega} (|z|^2 + |z_t|^2 + |\nabla z|^2) d\x.
	\label{E}
\end{equation}
 Define  a smooth cut-off function $\chi_1(t)$ satisfying
\[
	\chi_1(t) = \left\{
		\begin{array}{ll}
			0 & t < \theta,\\
			\in (0, 1) &\theta < t < 2\theta,\\
			1 & t > 2\theta.
		\end{array}
	\right.
\]
Multiplying $\chi_1(t)$ to both sides of \eqref{5.26}, we have
\begin{multline}
	(\chi_1(t) E(t))' 
	\leq  
	 C\chi_1(t)E(t) + \chi_1'(t) E(t) 
	 \\
	+ C\left( \int_{\Omega}|z(\x, 0)|^2 d\x
	+ \int_{\Omega}|\calL z|^2d\x
	+ \delta^2 + \alpha \|w^*\|_{H^3(\Omega \times [0, T])}^2\right).
	\label{5.27}
\end{multline}
Due to the trace theory, $\ds\int_{\Omega}|z(\x, 0)|^2d\x$ is bounded by $C\|z\|_{H^1(Q_{4\epsilon})}^2$.   We deduce from \eqref{5.27} that
\[
(\chi_1(t) E(t))' 
	\leq  
	 C\chi_1(t)E(t) + 
	C\left( \|z\|_{H^1(Q_{4\epsilon})}^2 
	+ |\chi_1'(t) E(t)|
	+ \int_{\Omega}|\calL z|^2d\x
	+ \delta^2 + \alpha \|w^*\|_{H^3(\Omega \times [0, T])}^2\right).
\]
Using Gr\"onwall's inequality and noting that $\chi_1(0) = 0$, we have
\[
	\chi_1(t)E(t) \leq C\left( \|z\|_{H^1(Q_{4\epsilon})}^2
	+ |\chi_1'(t) E(t)|
	+ \int_{\Omega}|\calL z|^2d\x
	+ \delta^2 + \alpha \|w^*\|_{H^3(\Omega \times [0, T])}^2\right) \quad t \in [0, T].
\]
Integrating the inequality above with respect to $t$, we have
\begin{align}
	\int_{\Omega \times [0, T]}  \chi_1(t)(|z|^2 + |z_t|^2 + |\nabla z|^2) dt d\x &\leq C\left( \|z\|_{H^1(Q_{4\epsilon})}^2
	+ \delta^2 + \alpha \|w^*\|_{H^3(\Omega \times [0, T])}^2  
	\right) \notag
	\\
	&\quad + C\int_{0}^T|\chi_1'(t) E(t)|dt
	+ C\int_{\Omega \times [0, T]}|\calL z|^2d\x dt.
	\label{5.28}
\end{align}
Since $\chi_1'(t) = 0$ for $t > 2\theta$ and $\Omega \times [0, 2\theta] \subset Q_{4\epsilon}$, it follows from \eqref{E} that
\begin{equation}
	\int_{0}^T|\chi_1'(t) E(t)|dt 
	= \int_{0}^{2\theta}|\chi_1'(t) E(t)|dt 
	\leq C\|z\|_{H^1(Q_{4\epsilon})}^2.
	\label{5.37}
\end{equation}
Using \eqref{step 1} and \eqref{5.37}, we deduce from \eqref{5.28} that
\begin{equation}
\int_{\Omega \times [0, T]}  \chi_1(t)(|z|^2 + |z_t|^2 + |\nabla z|^2) dt d\x \leq C\left( \|z\|_{H^1(Q_{4\epsilon})}^2
	+ \delta^2 + \alpha \|w^*\|_{H^3(\Omega \times [0, T])}^2  
	\right).
	\label{5.38}
\end{equation}
It follows from \eqref{5.25} and \eqref{5.38} that
\begin{multline}
	\int_{\Omega \times [0, T]}  \chi_1(t)(|\nabla z|^2 + z_t^2 + z^2) d\x dt  \leq 
	\frac{C\exp(-4\lambda \epsilon)}{\lambda }\int_{\Omega} \int_{\hat t_{2\epsilon}(\x)}^{\hat t_{\epsilon}(\x)}  (|\nabla z|^2 + z_t^2 + z^2) dt d\x 
	  \\
	  + C \exp(2\lambda (P_2^2 - 4\epsilon))\Big( \delta^2  
	+ \alpha \| w^*\|_{H^3(\Omega \times [0, T])}^2 \Big).
	\label{5.29}
\end{multline}
Adding \eqref{5.25} and \eqref{5.29}, we have
\begin{multline*}
	\int_{\Omega \times [0, T]}  (|\nabla z|^2 + z_t^2 + z^2) d\x dt  \leq 
	\frac{C\exp(-4\lambda \epsilon)}{\lambda }\int_{\Omega} \int_{\hat t_{2\epsilon}(\x)}^{\hat t_{\epsilon}(\x)}  (|\nabla z|^2 + z_t^2 + z^2) d\x dt 
	  \\
	  + C \exp(2\lambda (P_2^2 - 4\epsilon))\Big( \delta^2  
	+ \alpha \| w^*\|_{H^3(\Omega \times [0, T])}^2 \Big).
\end{multline*}
Fix $\lambda$ as a large number such that $C\exp(-4\lambda \epsilon) < 1/2$. We have 
\begin{equation}
	\int_{\Omega \times [0, T]}  (|\nabla z|^2 + z_t^2 + z^2) d\x dt  \leq 
	   C \Big( \delta^2  
	+ \alpha \| w^*\|_{H^3(\Omega \times [0, T])}^2 \Big).
	\label{5.30}
\end{equation}
Combining    \eqref{error function}, \eqref{z}, \eqref{5.4} and \eqref{5.30}, we obtain \eqref{5.6}. 

\noindent {\bf Step 6.}
	The analog of \eqref{2.7} for the function $p^*$ is
	\[
		p^*(\x) = c(\x) w^*(\x, 0) - \frac{ \Delta f(\x) + a(\x) f(\x) + B(\x) \cdot \nabla f(\x)}{h(\x, 0)}
	\] and the one for the function $p_{\delta, \alpha}$ is
	\[
		p_{\delta, \alpha}(\x) = c(\x) w_{\alpha}(\x, 0) - \frac{ \Delta f(\x) + a(\x) f(\x) + B(\x) \cdot \nabla f(\x)}{h(\x, 0)}
	\] 
	Using these identities, \eqref{5.6} and the trace theory, we obtain \eqref{Lipschitz}.

The proof is complete.


\begin{remark}
	Up to the knowledge of the author, the available convergence results and the rates of the convergence are known for partial differential equation without the presence of Volterra integral; see \cite{Klibanov:anm2015}. This is the first time when this method is extended for integro-differential equations.
\end{remark}

\section{Numerical tests} \label{sec num}

In this section, we display some numerical examples when $d = 2$. The set $\Omega$ is the cube $(-0.5, 0.5)^2$ and the time $T$ is 1. 

For the simplicity, we only consider the case $c(\x) = 1$, $a(\x) = B(\x) = 0$ for which condition \eqref{1.5} holds true.
 In addition, we set the initial value and velocity $f$ and $g$ of the wave field $u(\x, t)$ identically equal zero. 
These simplifications do not weaken the result of the paper because the contributions of those functions are not important in our analysis. 
 More precisely, in this section,  we implement our method and display numerical solutions to the ISP for the problem
\begin{equation*}
	\left\{
		\begin{array}{ll}
			u_{tt} = \Delta u + p(\x) h(\x, t) & (\x, t) \in \R^2 \times [0, \infty),\\
			u(\x, 0) = u_{t}(\x, 0) = 0 &\x \in \Omega.
		\end{array}
	\right.
\end{equation*}

\subsection{Generating the simulated data}

We solve the forward problem using the finite difference method. Set $\Omega_1 = (-R, R)^2$. In the computation, we choose $R = 3$. We create a $N \times N$, $N = 500$, grid 
\[
	\mathcal{G} = \{\x_{m, n} = (-R + (m - 1)d_{\x}, -R + (n - 1)d_{\x}): 1 \leq m, n \leq N\} \subset \Omega_1
\] 
where $d_{\x} = 2R/(N-1).$ We also split the time interval into a uniform partition 
\[
	t_1 = 0 < t_2 < \dots < t_{Nt} = T = 1
\] where $N_t = 120$. The step size of the time variable is $d_t = t_2 - t_1 = T/N_t.$
Then, given $u(\x, 0) = u(\x, d_t) = 0$ for all $\x \in G$, we explicitly calculate $u(\x, t + d_t)$ by
\[
	u(\x, t + d_t) = 2 u(\x, t) - u(\x, t - d_t) + d_t^2(\Delta u(\x, t) + p(\x) h(\x, t)).
\]
The noiseless data $F^*(\x, t) = u(\x, t)$ and $G^*(\x, t) = \partial_{\nu} u(\x, t)$ for $\x \in \partial \Omega$ and $t \in \{t_1, \dots, t_{Nt}\}$ can be extracted easily.

Let $\delta > 0$ denote the noise level. Noisy data are set to be
\begin{equation}
	F(\x, t) = F^*(\x, t)(1 + \delta (2\textrm{rand}(\x, t) - 1))
	\label{fnoise}
\end{equation}
and
\begin{equation}
	G(\x, t) = G^*(\x, t)(1 + \delta (2 \textrm{rand}(\x, t) - 1)).
	\label{gnoise}
\end{equation}
Here, $\textrm{rand}$ is a Matlab function that generates uniformly distributed random numbers in the interval $[0, 1]$. We use  $2 \textrm{rand}(\x, t) - 1$ in \eqref{fnoise} and \eqref{gnoise} to create random numbers in the interval $[-1, 1].$
In this paper, we test our method when $\delta$ takes values  $2\%$, $5\%$ and $10\%.$

\subsection{Differentiating noisy data by Tikhonov regularization} \label{sec differentiate}

The first step to solve the ISP is to calculate the second derivatives of $F(\x, t)$ and $G(\x, t)$ with respect to $t$. 
Since data are supposed to contain noise, see \eqref{fnoise} and \eqref{gnoise},  they can not be differentiated  using the finite difference method. 
We calculate the second derivative of data by the Tikhonov regularization method, which is widely used in the community.  For each $\x \in \partial \Omega$, noting that 
\begin{equation*}
	F(\x, 0) = F_t(\x, 0) = 0 
\end{equation*}
we can write
\begin{equation}
	F(\x, t_n) = \int_0^{t_n} \int_0^s F_{tt}(\sigma) d\sigma ds
	= d_t^2 \sum_{i = 1}^{n} \sum_{j = 1}^i F_{tt}(\x, t_j), \quad n = 1, \dots, N_t.
	\label{6.4}
\end{equation}
All equations in \eqref{6.4} constitute a linear system, say 
\begin{equation}
\F(\x) = A \F_{tt}(\x)
\end{equation}
where $\F$ and $\F_{tt}$ are the vectors $(F(\x, t_n))_{n = 1}^{N_t}$ and $(F_{tt}(\x, t_n))_{n = 1}^{N_t}$ respectively and the matrix $A$ is given by
\begin{equation}
	A = \left[
		\begin{array}{cccc}
			1&0&\dots&0\\
			2&1&\dots&0\\
			\vdots&\vdots&\ddots&\vdots\\
			N_t & N_t -1 &\dots&1
		\end{array}
	\right].
\end{equation}
The vector $\F_{tt}(\x)$ is approximated by the solution of
\begin{equation*}
	(A^T A + \varepsilon \Id )\F_{tt}(\x) = \F(\x)
\end{equation*} where $\Id$ is the identity matrix
 and $\varepsilon > 0$ is a small number.
In the computation, we choose $\epsilon = 10^{-5}.$

\subsection{Implementation}\label{sec Implementation}
 
In the previous section, we solved the forward problem in the domain $\Omega_1$, which was spitted to $500 \times 500$ grid. The mesh restricted on $\Omega = (-0.5, 0.5)^2 \Subset \Omega_1$ is of size $85 \times 85.$ We thus reset $N = 85.$ 
The meshgrid for $\Omega \times [0, T]$ is the set $\{(\x_{m, n}), t_j\}_{1 \leq m, n \leq N, 1 \leq j \leq N_t}$. Recall that $N_t = 120.$
Without confusing, for any function $\phi$ defined in $\Omega \times [0, T]$, we identify $\phi$ by the tensor $(\phi(\x_{m, n}, t_j))_{1 \leq m, n \leq N, 1 \leq j \leq N_t}.$

We establish the tensor $\mathcal{L}$ so that for all function $\phi$, the function $\phi_{tt} - \Delta \phi - \phi(0) \frac{h_{tt}}{\tilde h}$ is approximated by $L \phi$. Based on finite difference, the entries of the tensor $\mathcal{L}$ is given by 
\begin{equation*}
\begin{array}{lll}
	\mathcal{L}_{m, n, j; m, n, j} = -\frac{2}{d_t^2}  + \frac{4}{d_{\x}^2} &\mathcal{L}_{m, n, j; m, n, j+1} = -\frac{1}{d_t^2} &\mathcal{L}_{m, n, j; m, n, j-1} = -\frac{1}{d_t^2}\\
	\mathcal{L}_{m, n, j; m-1, n, j} =\frac{1}{d_{\x}^2} &\mathcal{L}_{m, n, j; m+1, n, j} =\frac{1}{d_{\x}^2} &\mathcal{L}_{m, n, j; m, n - 1, j} =\frac{1}{d_{\x}^2}\\
	\mathcal{L}_{m, n, j; m, n + 1, j} =\frac{1}{d_{\x}^2} &\mathcal{L}_{m, n, j; m, n, 1} = -\frac{h_{tt}(\x_{m, n}, t_j)}{\tilde h(\x_{m, n}, t_j)}
\end{array}
\end{equation*}
for $2 \leq m, n \leq N - 1$ and $2 \leq j \leq N_t-1$. The other entries of $\calL$ are 0.

We next implement the constraints in \eqref{initial vel} with $\Psi = 0$ and \eqref{boundary}. 
\begin{enumerate}
	\item The constraint $w_{t}(\x, 0) = 0$ can be understood in the sense of finite difference as $\mathcal{D}_t w = 0$ where $(\mathcal{D}_t)_{m,n,1} = -\frac{1}{d_t}$, $(\mathcal{D}_t)_{m,n,2} = \frac{1}{d_t}$ and the other entries of $\mathcal{D}_t$ are $0$.
	\item The constraint $w(\x, t) = \zeta(\x, t)$ can be written as $\mathcal{D} w= \zeta(\x, t)$ where 
\begin{equation*}
\mathcal{D}_{m, n, j; m, n, j } = 
	\left\{ 
		\begin{array}{ll}
			1 & m \in \{1, N\} \mbox{ or } n \in \{1, N\}\\
			0 &\mbox{otherwise.}
		\end{array}		
	\right. 
\end{equation*} Here, we need to extend $\zeta(\x, t) = 0$ inside $\Omega.$
\item The constraint $\partial_{\nu} w(\x, t) = \xi(\x, t)$ can be written as $\mathcal{N} w= \xi(\x, t)$ where 
\begin{equation*}
\begin{array}{llr}
	\mathcal{N}_{1,n, j} = \frac{1}{d_{\x}} &\mathcal{N}_{2,n, j} = -\frac{1}{d_{\x}} &\quad 1 \leq n \leq N,\\
	\mathcal{N}_{N,n, j} = \frac{1}{d_{\x}} &\mathcal{N}_{N-1,n, j} = -\frac{1}{d_{\x}}& \quad 1 \leq n \leq N,\\
	\mathcal{N}_{m,1, j} = \frac{1}{d_{\x}}, &\mathcal{N}_{m,2, j} = -\frac{1}{d_{\x}} &\quad 1 \leq m \leq N,\\
	\mathcal{N}_{m,N, j} = \frac{1}{d_{\x}}, &\mathcal{N}_{m,N - 1, j} = -\frac{1}{d_{\x}} &\quad 1 \leq m \leq N
\end{array}	
\end{equation*}
The other entries of $\mathcal{N}$ are 0.
 Here, we need to extend $\xi(\x, t) = 0$ inside $\Omega.$
\end{enumerate}

Similarly, we  approximate the derivatives of the function $\phi$ by $D_x$, $D_y$ and $D_t$.  Here,
\begin{equation*}
	\begin{array}{rclrclr}
		(D_x)_{m,n,j;m,n,j} &=& -\frac{1}{d_{\x}} &(D_x)_{m,n,j;m+1,n,j} &=& -\frac{1}{d_{\x}} &1 \leq m \leq N-1\\
		(D_y)_{m,n,j;m,n,j} &=& -\frac{1}{d_{\x}} &(D_x)_{m,n,j;m,n+1,j} &=& -\frac{1}{d_{\x}} &1 \leq n \leq N - 1\\
		(D_t)_{m,n,j;m,n,j} &=& -\frac{1}{d_t} &(D_t)_{m,n,j;m,n,j+1} &=& -\frac{1}{d_t} &1 \leq j \leq N_t - 1
	\end{array}
\end{equation*}
Other entries of $D_x, D_y, D_t$ are 0.

It is convenient to identify the $6-$order tensors $\mathcal{L, D}$ and $\mathcal{N}$ above by matrices by assigning the triple index $m, n, j$ the single index
\begin{equation*}
	\mathfrak{i}  = (m-1)N N_t + (n - 1) N_t + j. 
\end{equation*}
 By this, instead of approximating a function $\phi$ by a $3-$order tensor $(\phi(\x_{m, n}, t_j))$, we consider function $\phi$ by the vector $\phi(\xi_{\mathfrak{i}}).$
This ``line up" technique is employed because operators on multi-dimensional tensors are not supported in Matlab. 

Hence a function $w$ satisfies equation \eqref{problem for w}, \eqref{initial vel} and \eqref{boundary} if and only if
\begin{equation}
	\mathcal{C} w = \mathfrak{b}
		\label{6.13}
\end{equation}
where 
\begin{equation*}
	\mathcal{C} = \left[
		\begin{array}{c}
			\mathcal{D}_t\\
			\mathcal{D}\\
			\mathcal{N}\\
			\mathcal{L}
		\end{array}
	\right] \quad \mathfrak{b} = \left[
		\begin{array}{c}
			\vec{0}\\
			\zeta\\
			\xi\\
			\vec{0}
		\end{array}
	\right], \quad \mbox{and } \vec{0} \mbox{ is the zero vector in } \R^{N^2 N_t}.
\end{equation*}
Since the data $F$ and $G$ are noisy, \eqref{6.13} might have no solution. Instead, we solve
\begin{equation}
	(\mathcal{C}^T \mathcal{C} + \varepsilon_1 \Id + \varepsilon_2 (D_x^TD_x + D_y^TD_y + D_t^TD_t))w = \mathfrak{b}.
	\label{6.14}
\end{equation}
In our computation $\varepsilon_1 = 3\cdot10^{-3}$ and $\varepsilon_2 = 1.5\cdot10^{-4}.$

 \begin{remark}
 	Solving \eqref{6.14} is somewhat equivalent to finding the regularized solution of \eqref{problem for w}, \eqref{initial vel} and \eqref{boundary}. 
	In fact, the minimizer of $\J_{\alpha}$ is the solution of the equation $D \J_{\alpha} w = 0$ and $D\J_{\alpha} w$ is actually $\calL^T \calL + \alpha I^T I$ where $|Iw|$ gives the $H^3$ norm of of $w$. Here are some differences of the numerical implementation from the analysis in Sections \ref{sec QR} and \ref{sec regu}:
	\begin{enumerate}
	\item A small difference of solving \eqref{6.14} and minimizing $\J_{\alpha}$ is that we reduce the $H^3$ norm to be $H^1$ norm in the regularization term because this regularization term is already provide good reconstruction results.
	\item In theory, we set the regularization parameter by a single number $\alpha$ but in computation, we choose the regularization parameters as $\varepsilon_1$ and $\varepsilon_2$ above by a trial and error procedure. 
	\item The more important fact is that since \eqref{6.14} is uniquely solvable, it is not necessary to verify the condition $K$ is non-empty. See Remark \ref{rem 4.1}.
	\end{enumerate}
 \end{remark}
 
\subsection{Numerical examples}

To illustrate the efficiency of the method, we display here some numerical results. Set
\[
	h(\x, t) = 1 + \exp(-(4 + |\x|^2) t) \quad (\x, t) \in \R^d \times [0, T].
\]

We test four (4) models, described below. In Figures \ref{fig model 1}--\ref{fig model 4}, the notation $p_{\rm comp}$ indicates the computed source functions.

\noindent {\bf Test 1.} The true function $p^*(\x)$ is given by
	\[
		p^*({\x}) = \left\{
			\begin{array}{ll}
				1 & \mbox{if }4(x-0.15)^2 + y^2/8 \leq 0.1^2,\\
				-1&\mbox{if} \max\{|x + 0.15|, |y|\} < 0.1,\\
				0 &\mbox{otherwise}.			
			\end{array}
		\right.
	\] With this function, the source involves both negative and positive parts. The reconstructed results with various noise levels are displayed in Figure \ref{fig model 1}. 
\begin{figure}[]
	\begin{center}
	\subfloat[The function $p_{\rm true}$]{\includegraphics[width = 0.42\textwidth]{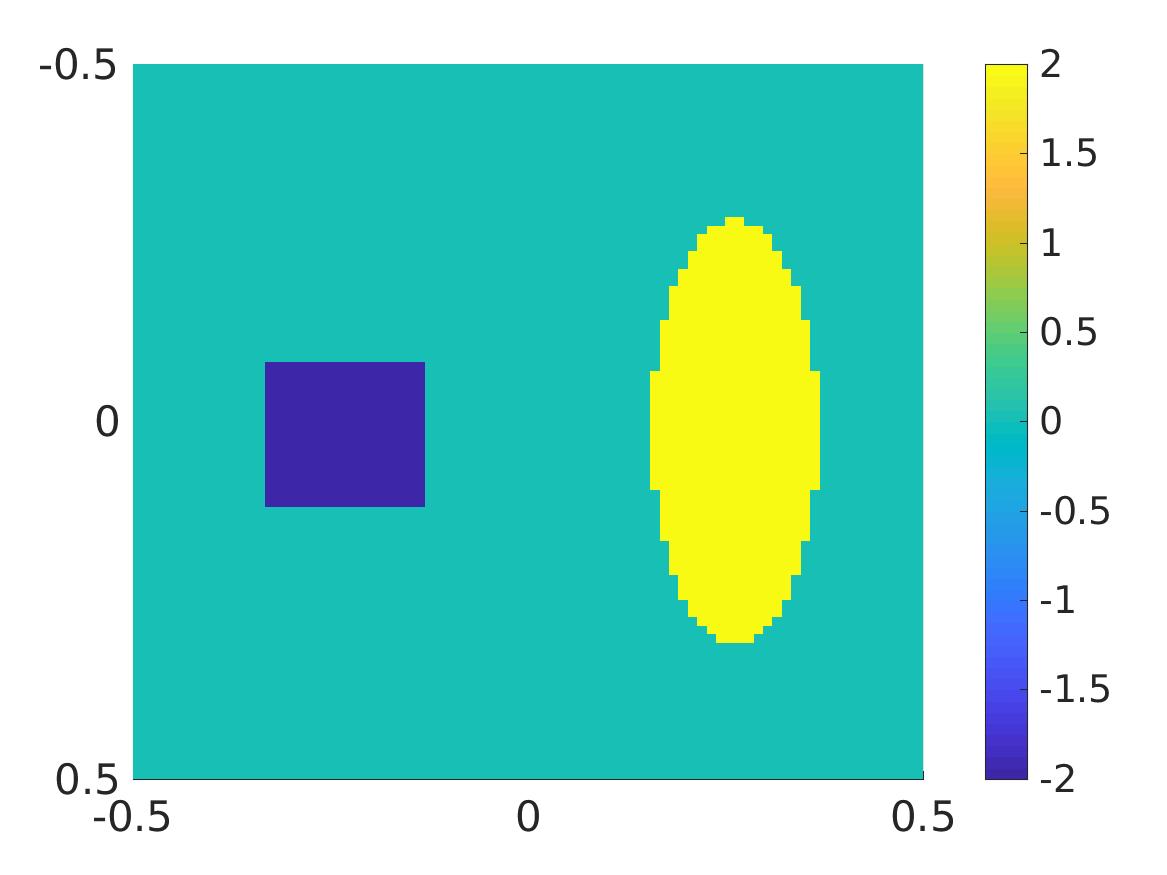}}\quad 
	\subfloat[The function $p_{\rm comp}$ when $\delta = 2\%.$]{\includegraphics[width = 0.42\textwidth]{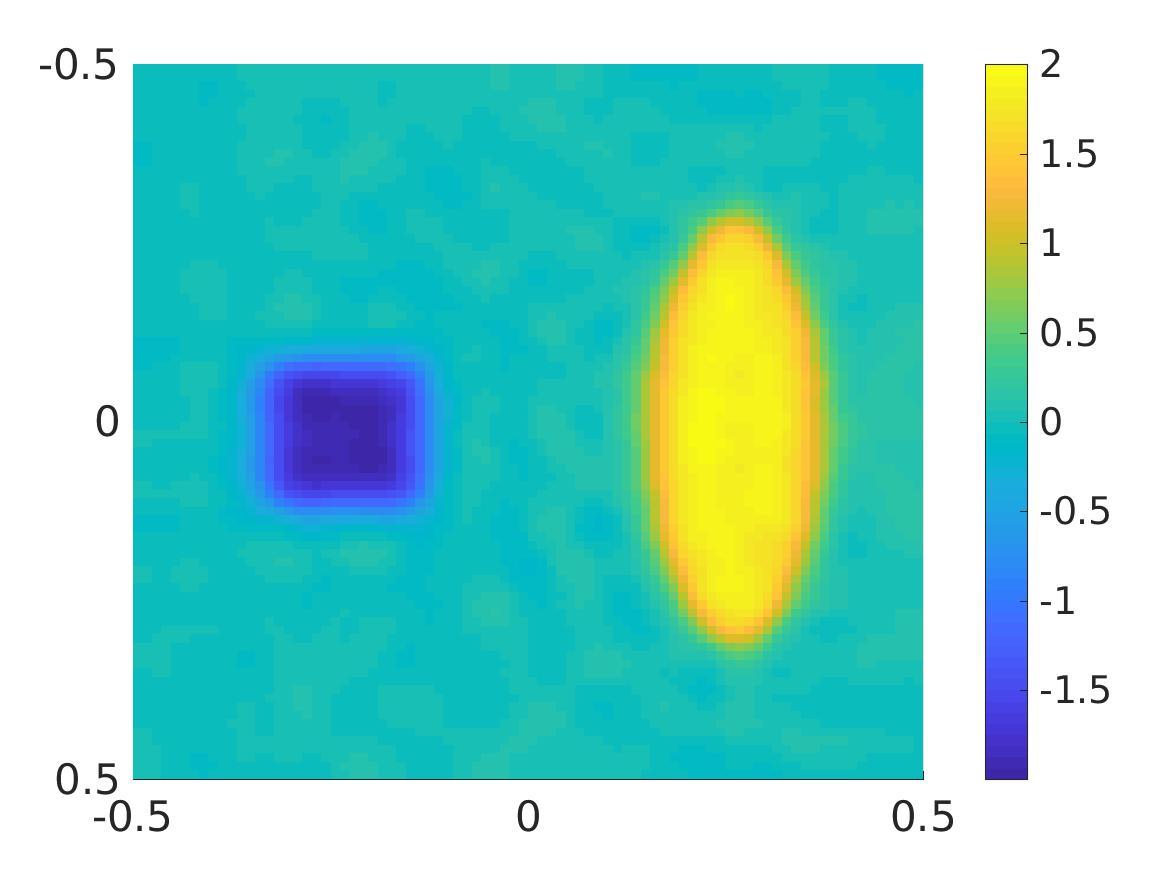}} 
	
		\subfloat[The function $p_{\rm comp}$ when $\delta = 5\%.$]{\includegraphics[width = 0.42\textwidth]{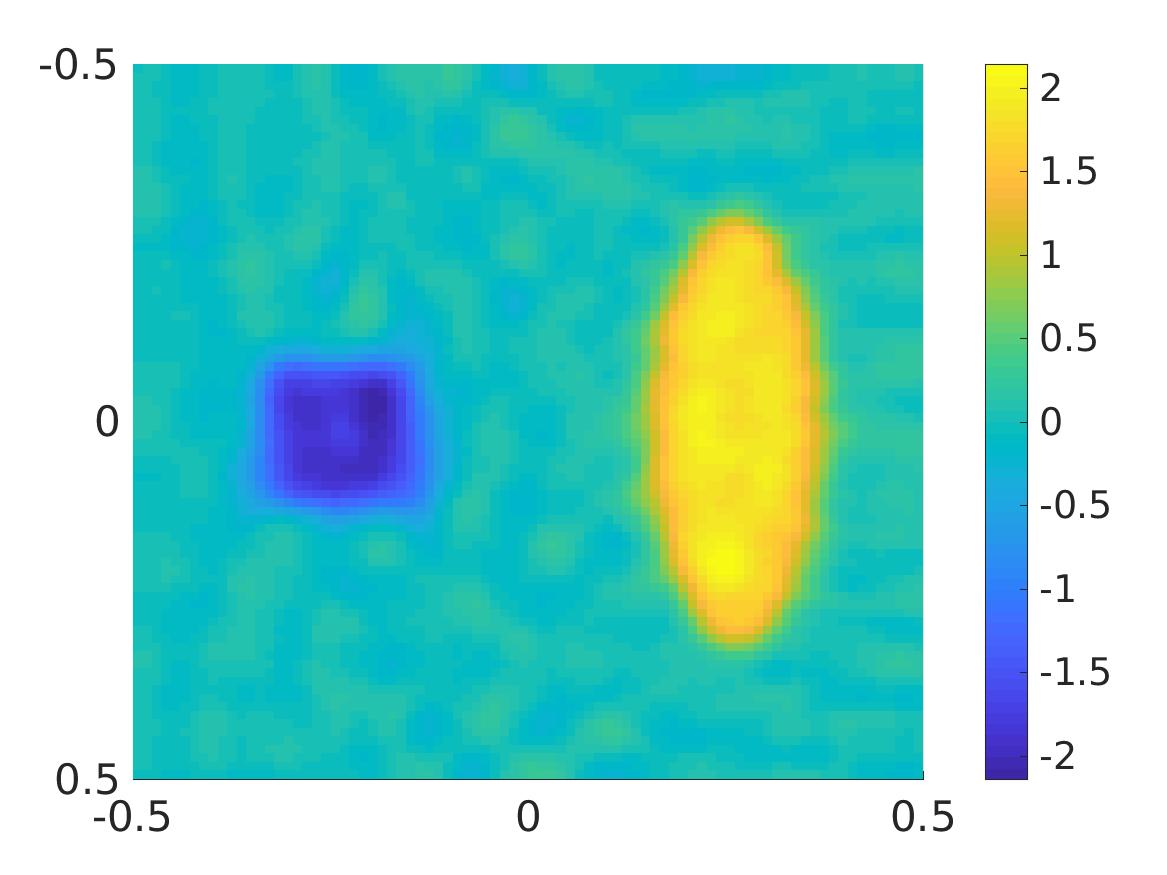}} \quad
	\subfloat[The function $p_{\rm comp}$ when $\delta = 10\%.$]{\includegraphics[width = 0.42\textwidth]{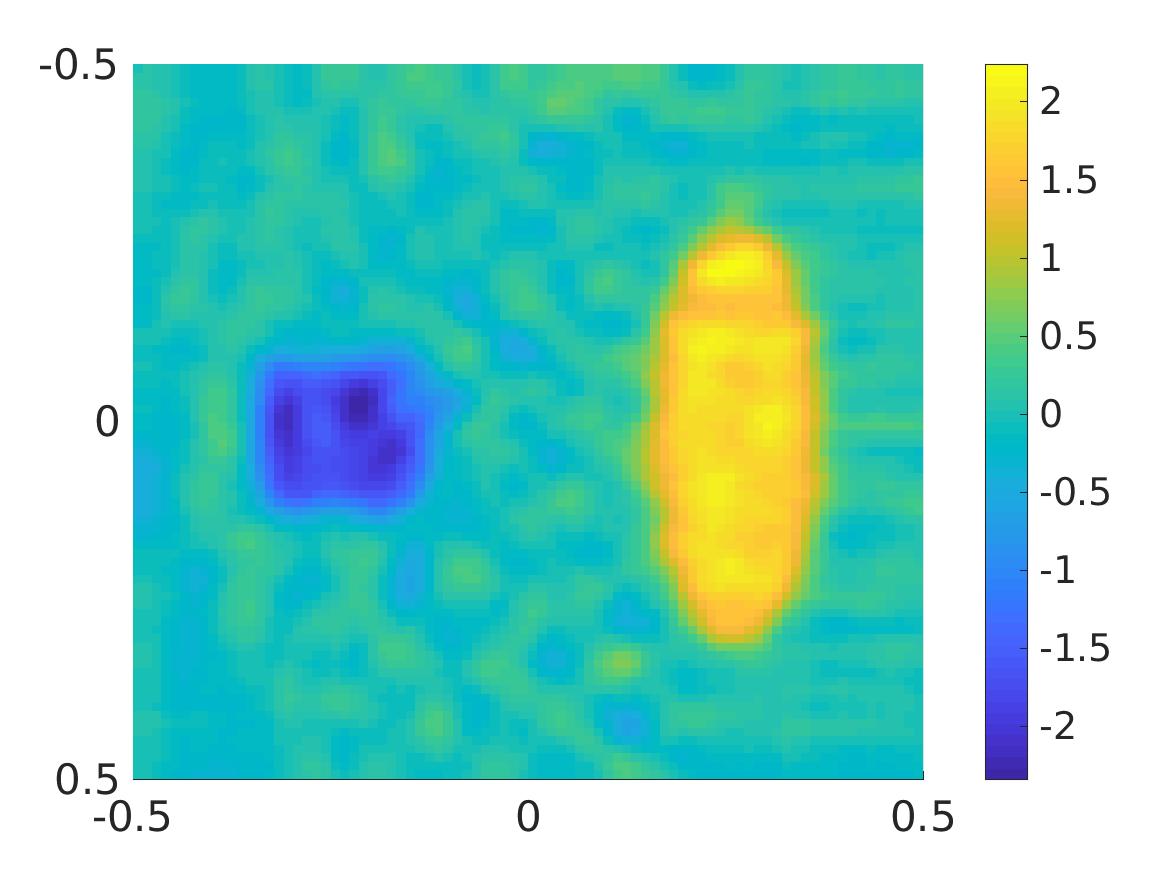}}
	
	\subfloat[The functions $p_{\rm comp}$ (dash-dot) and $p_{\rm true}$ (solid) on the line $y = 0$ when $\delta = 5\%$]{\includegraphics[width = 0.42\textwidth]{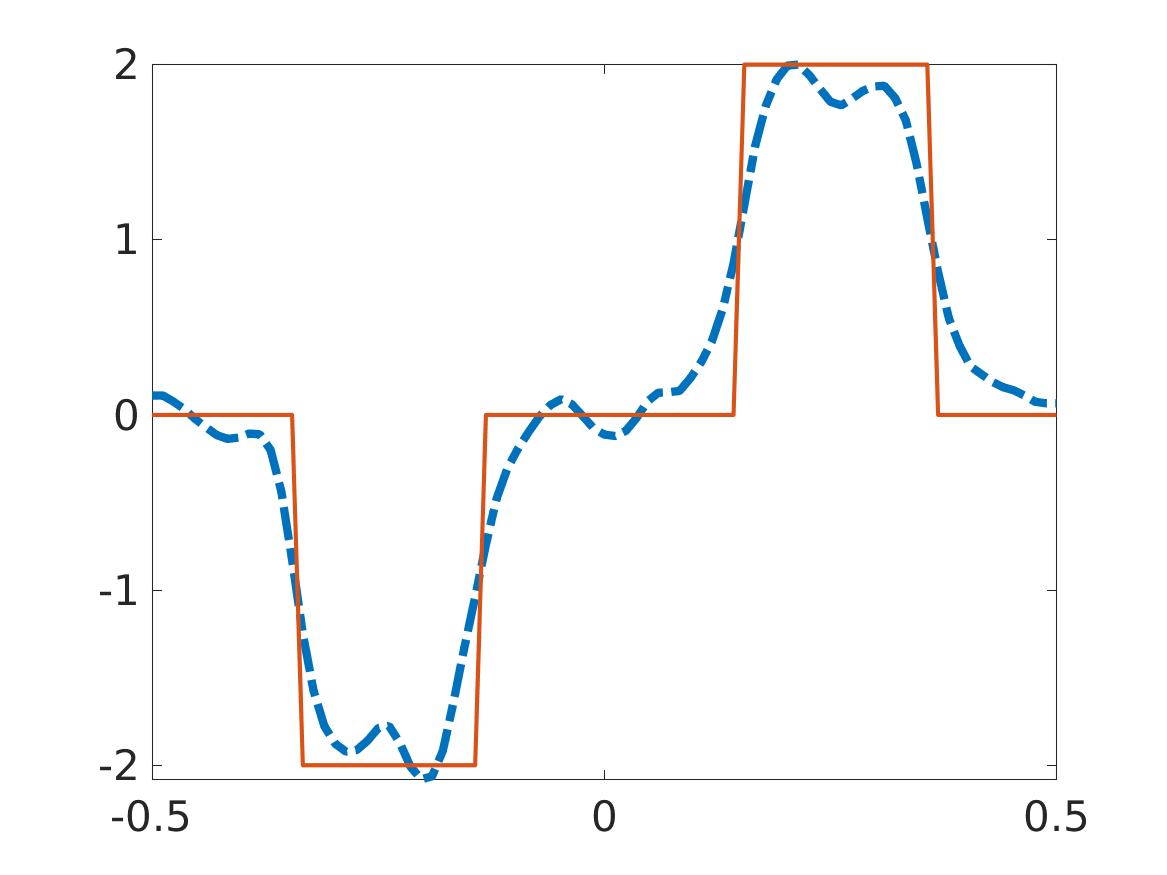}}  \quad
	\subfloat[The functions $p_{\rm comp}$ (dash-dot) and $p_{\rm true}$ (solid) on the line $y = 0$ when $\delta = 10\%$]{\includegraphics[width = 0.42\textwidth]{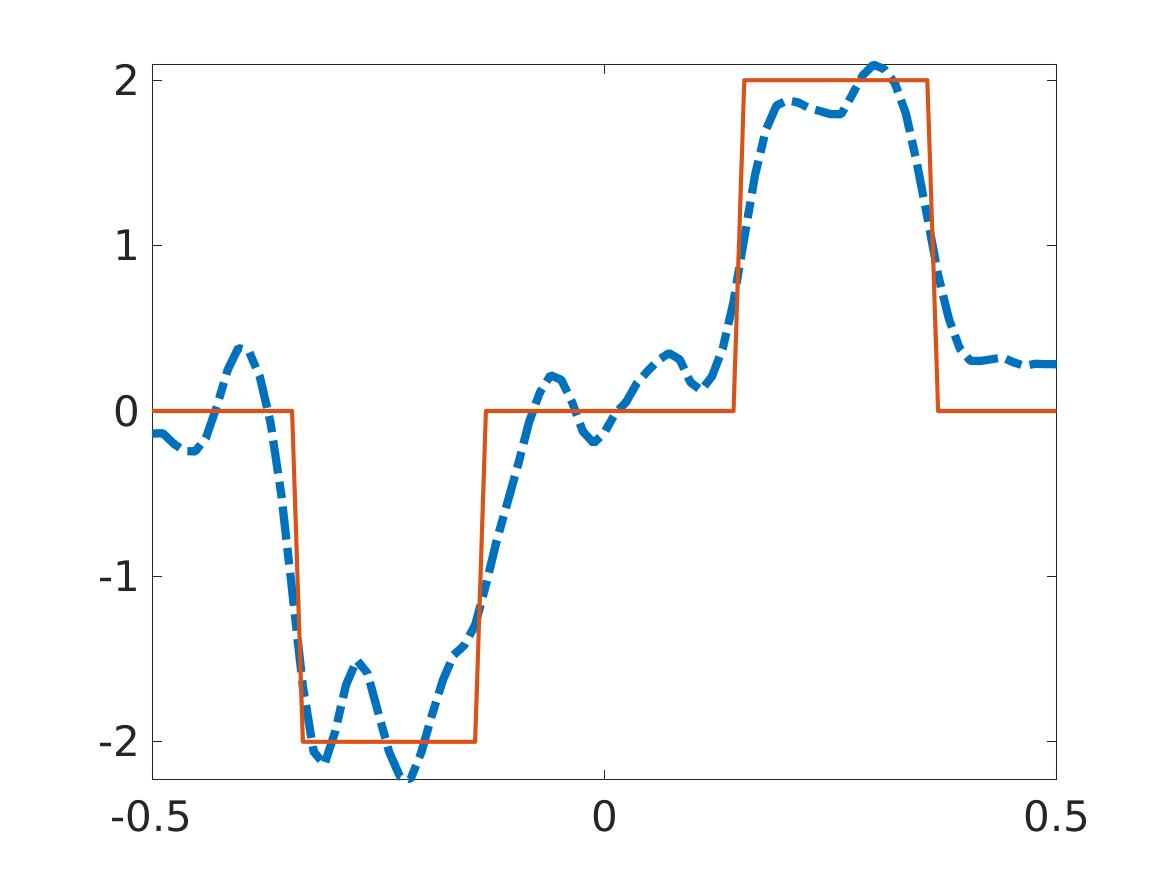}}
	\caption{\label{fig model 1} \it The true and reconstructed source functions in Test 1 where the data is noisy. It is evident that the shape and values of the source can be reconstructed with a high accuracy.}
	\end{center}
\end{figure}

\noindent {\bf Test 2.} The true function $p^*(\x)$ is given by
	\[
		p^*({\x}) = \left\{
			\begin{array}{rl}
				1.5 & \mbox{if } (x - 0.25)^2 + y^2 < 0.12^2,\\
				1&\mbox{if } 4x^2 + (y + 0.25)^2 < 0.15^2,\\
				-1&\mbox{if } |x+0.25| + |y| < 0.17,\\
				0 &\mbox{otherwise}.			
			\end{array}
		\right.
	\] With this function, the source involves three ``inclusions" with different values. The reconstructed results with various noise levels are displayed in Figure \ref{fig model 2}. 
\begin{figure}[]
	\begin{center}
	\subfloat[The function $p_{\rm true}$]{\includegraphics[width = 0.42\textwidth]{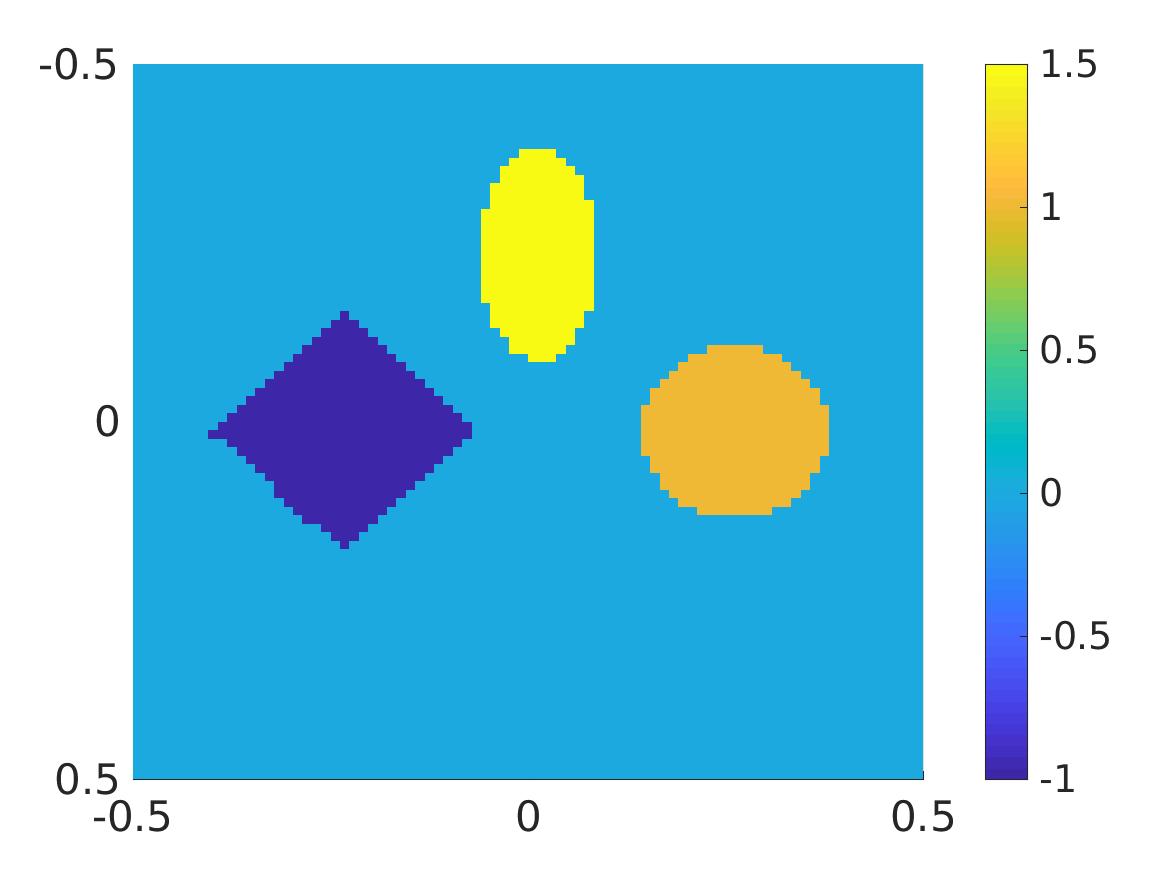}} \quad
	\subfloat[The function $p_{\rm comp}$ when $\delta = 2\%.$]{\includegraphics[width = 0.42\textwidth]{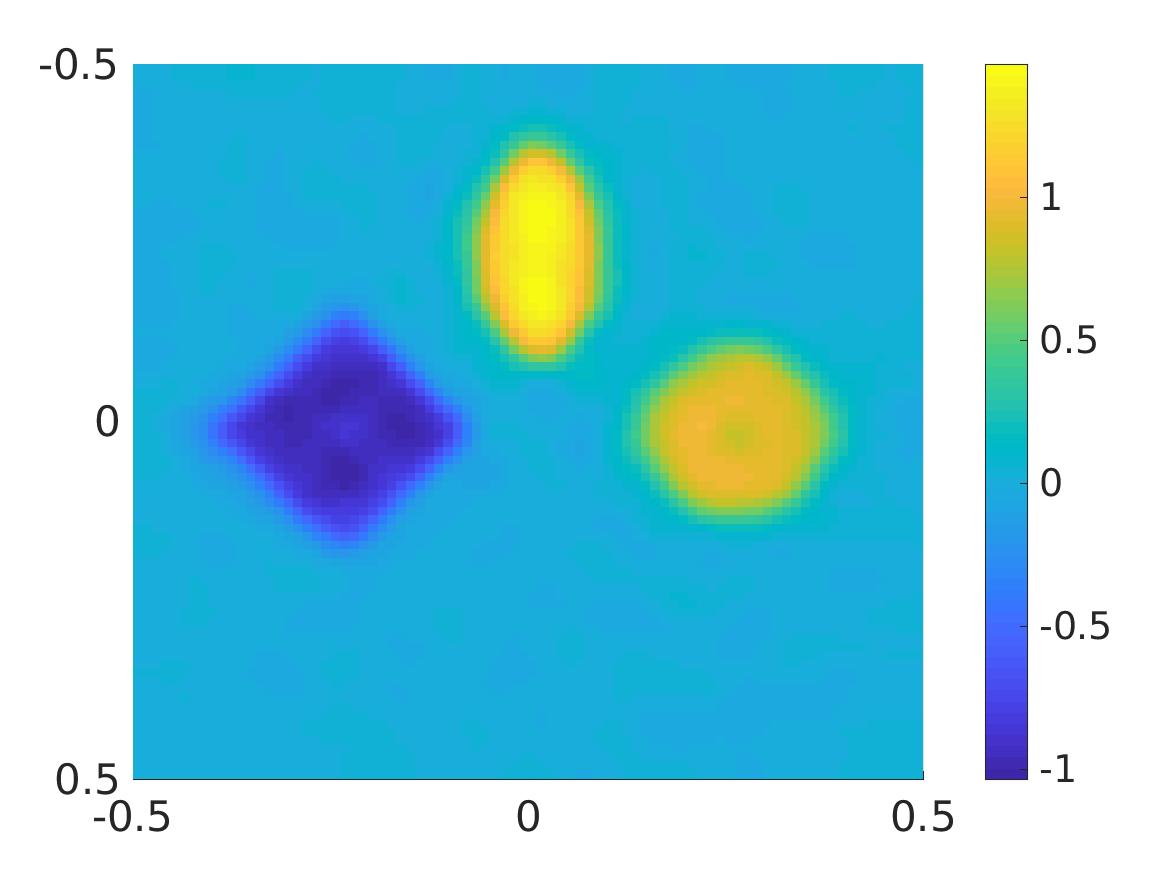}} 
	
		\subfloat[The function $p_{\rm comp}$ when $\delta = 5\%.$]{\includegraphics[width = 0.42\textwidth]{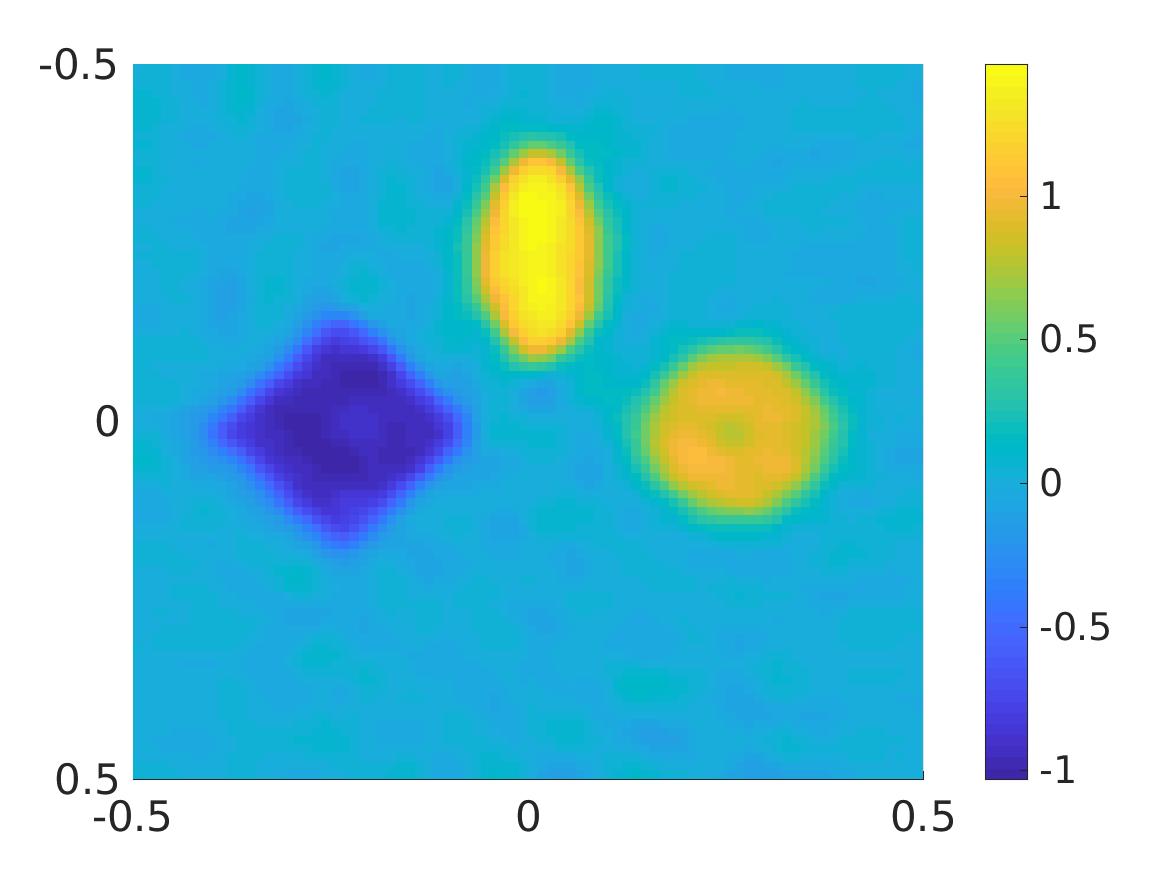}} \quad
	\subfloat[The function $p_{\rm comp}$ when $\delta = 10\%.$]{\includegraphics[width = 0.42\textwidth]{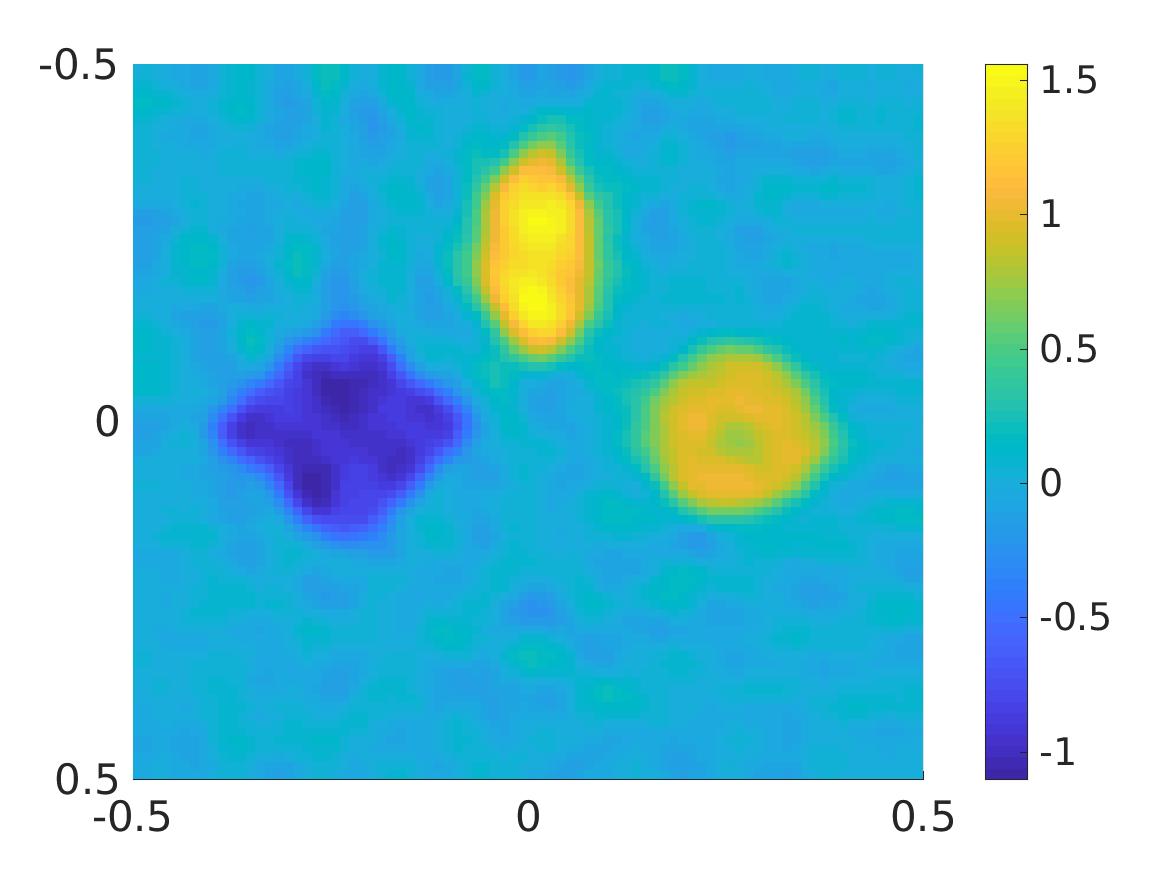}}
	
	\subfloat[The functions $p_{\rm comp}$ (dash-dot) and $p_{\rm true}$ (solid) on the line $y = 0$ when $\delta = 5\%$]{\includegraphics[width = 0.42\textwidth]{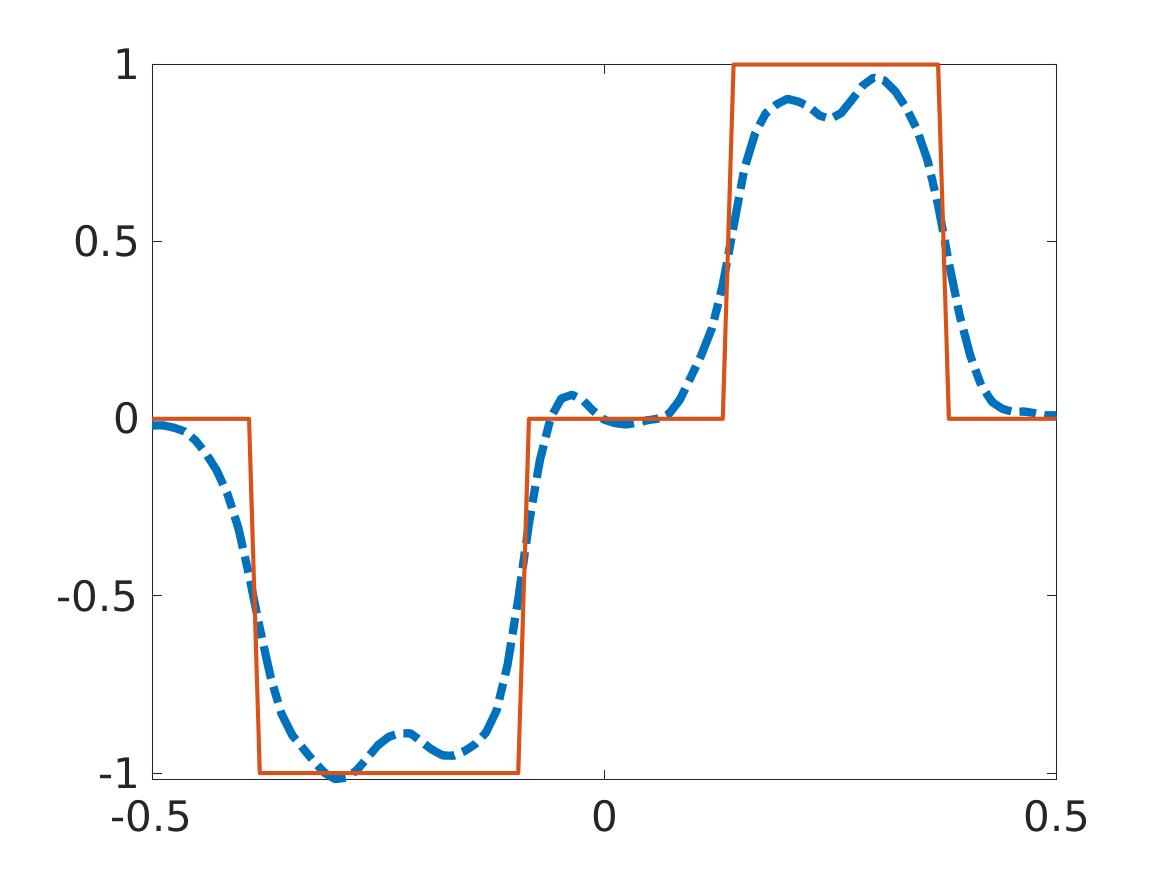}}  \quad
	\subfloat[The functions $p_{\rm comp}$ (dash-dot) and $p_{\rm true}$ (solid) on the line $y = 0$ when $\delta = 10\%$]{\includegraphics[width = 0.42\textwidth]{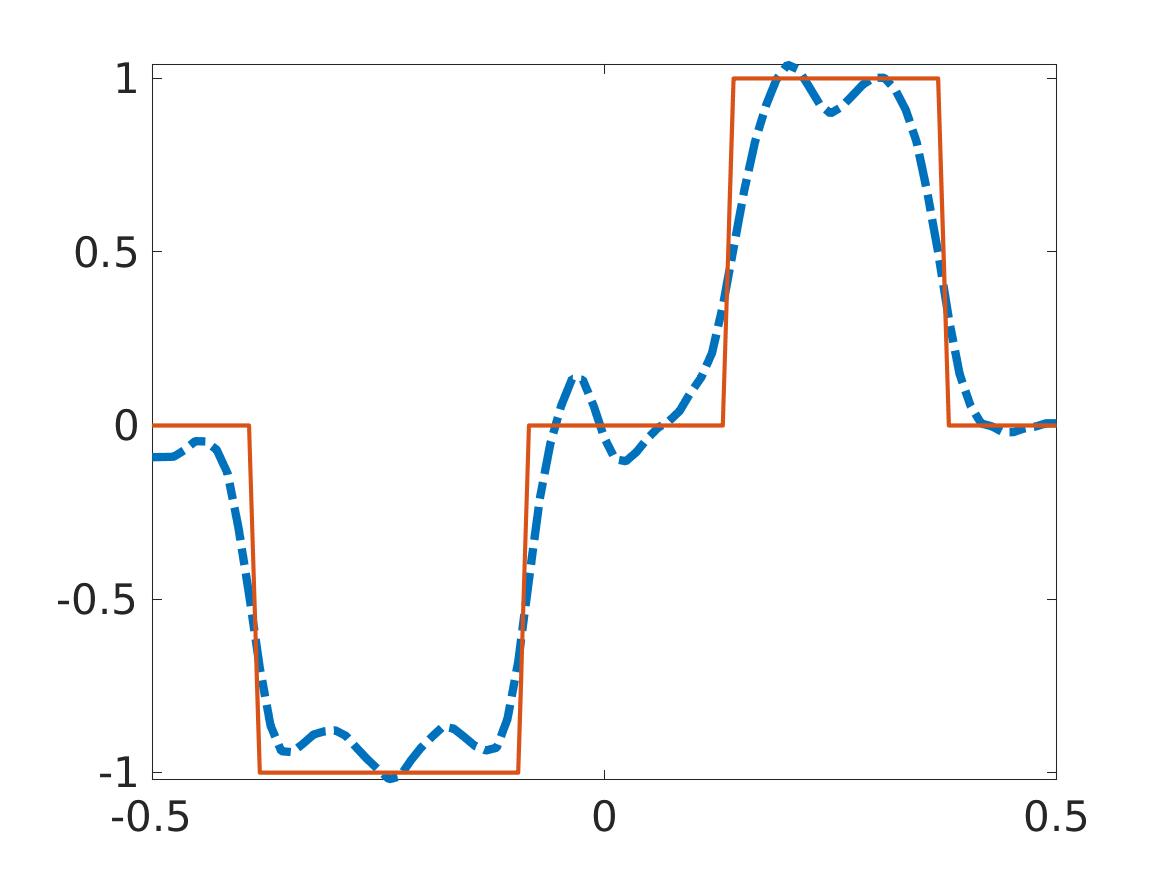}}
	\caption{\label{fig model 2} \it The true and reconstructed source functions in Test 2 where the data is noisy. Similarly to the case of Test 1, the shape and values of the source are reconstructed very well.}
	\end{center}
\end{figure}

\noindent {\bf Test 3.} The true function $p^*(\x)$ is a smooth function given by
\begin{equation*}
	p^*(\x) = 3(1-x)^2 \exp(-x^2 - (y+1)^2) 
   \\- 10\left(\frac{x}{5} - x^3 - y^5\right)\exp(-x^2-y^2) 
   - \frac{1}{3}\exp(-(x+1)^2 - y^2)
\end{equation*}
   for $\x = (x, y) \in \Omega.$ This function is named as ``peaks" in Matlab, which has several local maxima and minima. The reconstructed results are displayed in Figure \ref{fig model 3}.   
\begin{figure}[]
	\begin{center}
	\subfloat[The function $p_{\rm true}$]{\includegraphics[width = 0.42\textwidth]{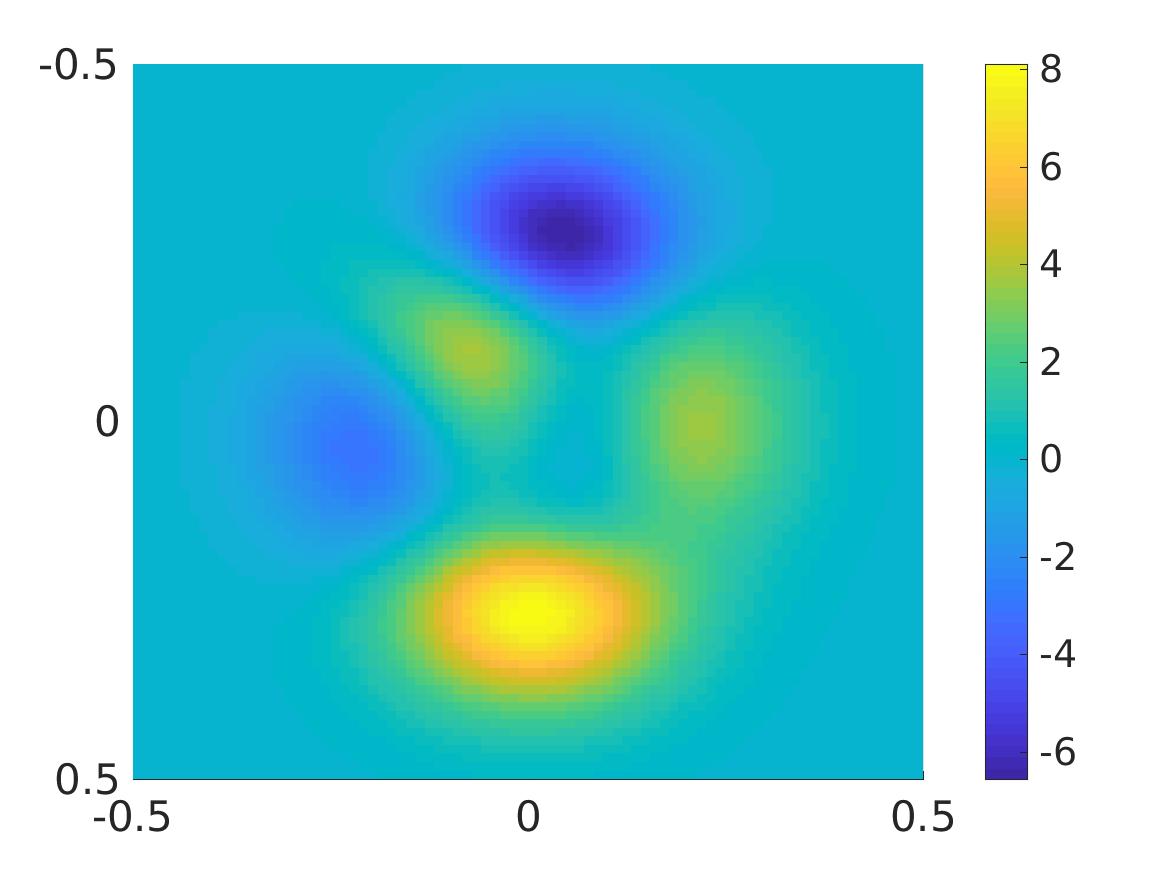}} \quad
	\subfloat[The function $p_{\rm comp}$ when $\delta = 2\%.$]{\includegraphics[width = 0.42\textwidth]{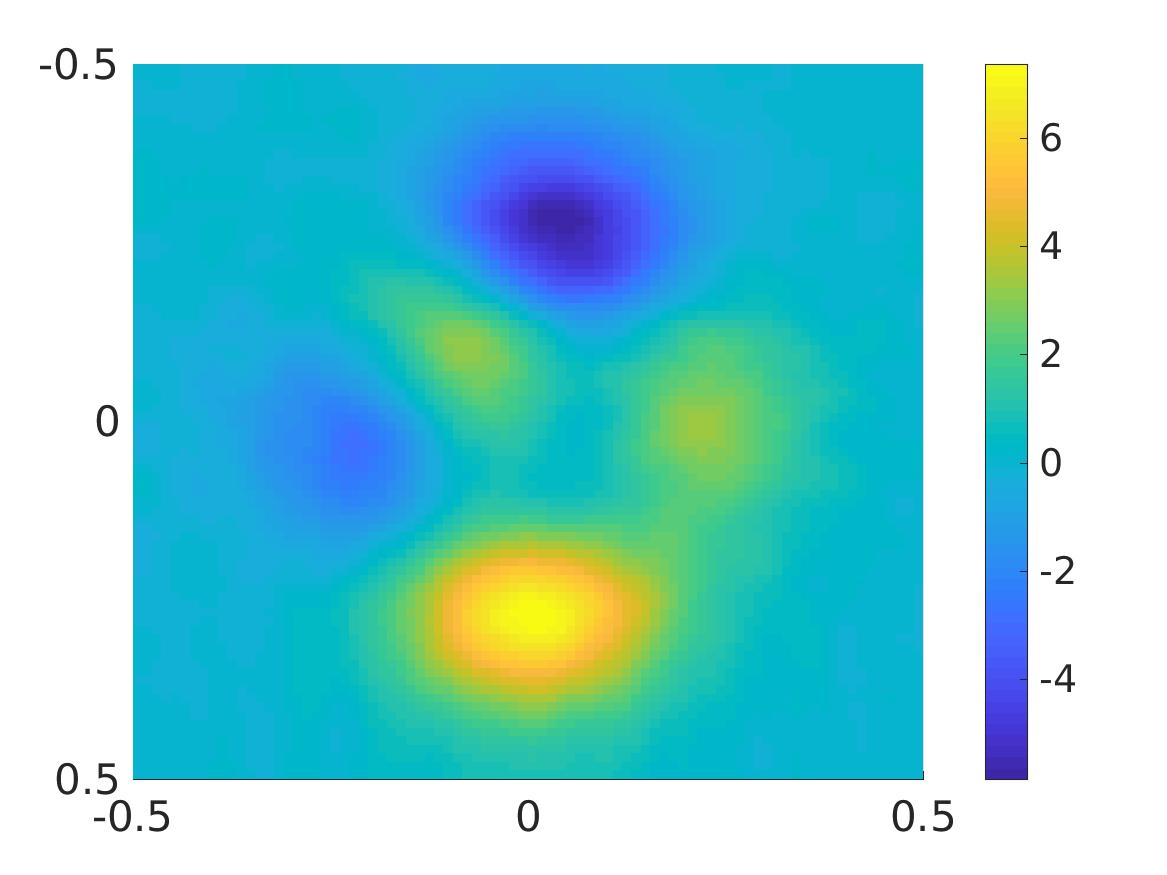}} 
	
		\subfloat[The function $p_{\rm comp}$ when $\delta = 5\%.$]{\includegraphics[width = 0.42\textwidth]{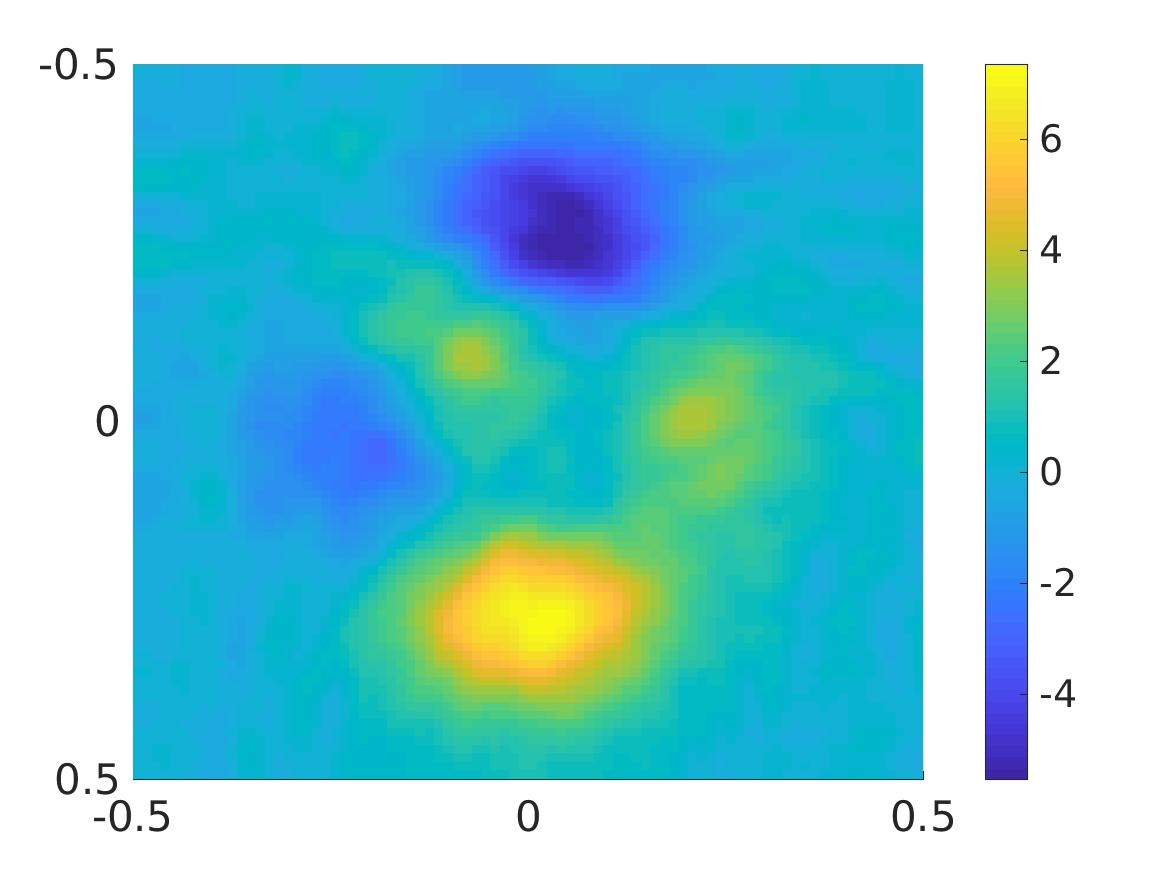}} \quad
	\subfloat[The function $p_{\rm comp}$ when $\delta = 10\%.$]{\includegraphics[width = 0.42\textwidth]{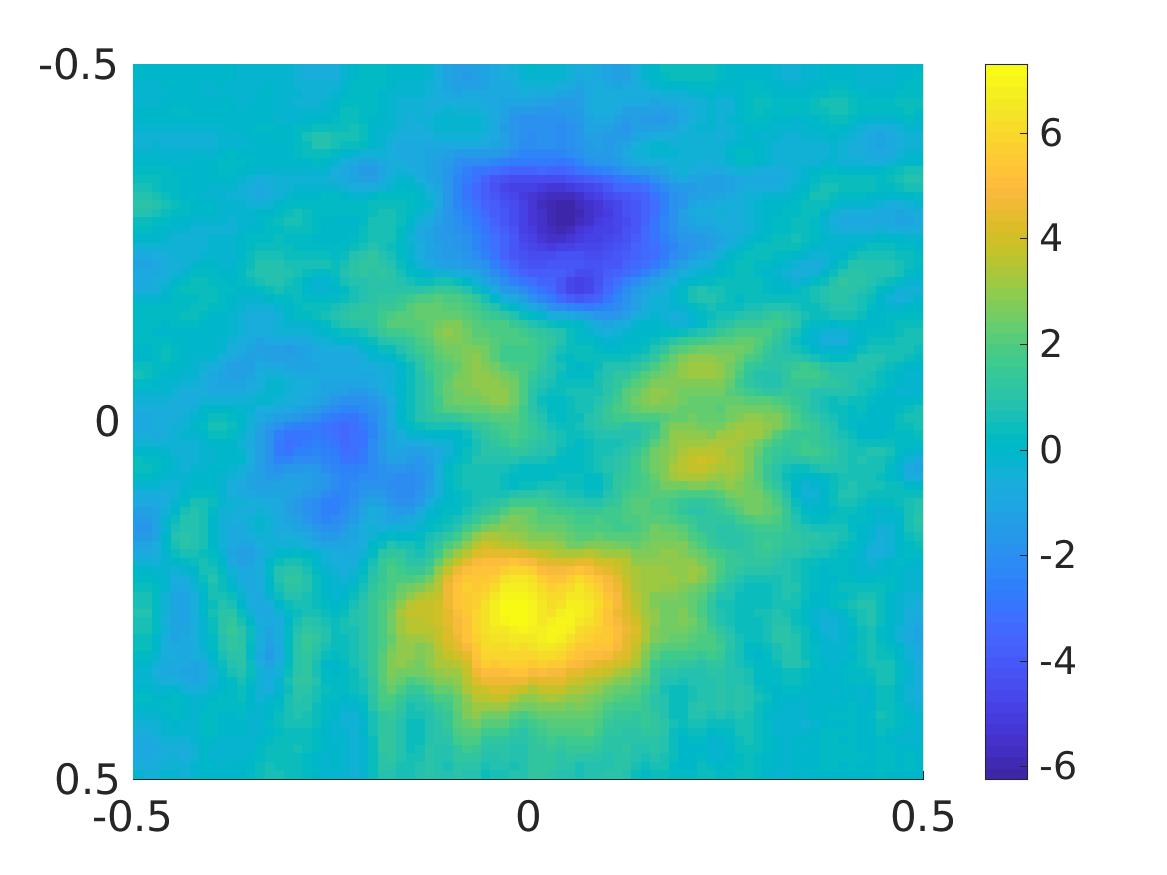}}
	
	\subfloat[The functions $p_{\rm comp}$ (dash-dot) and $p_{\rm true}$ (solid) on the line $y = 0$ when $\delta = 5\%$]{\includegraphics[width = 0.42\textwidth]{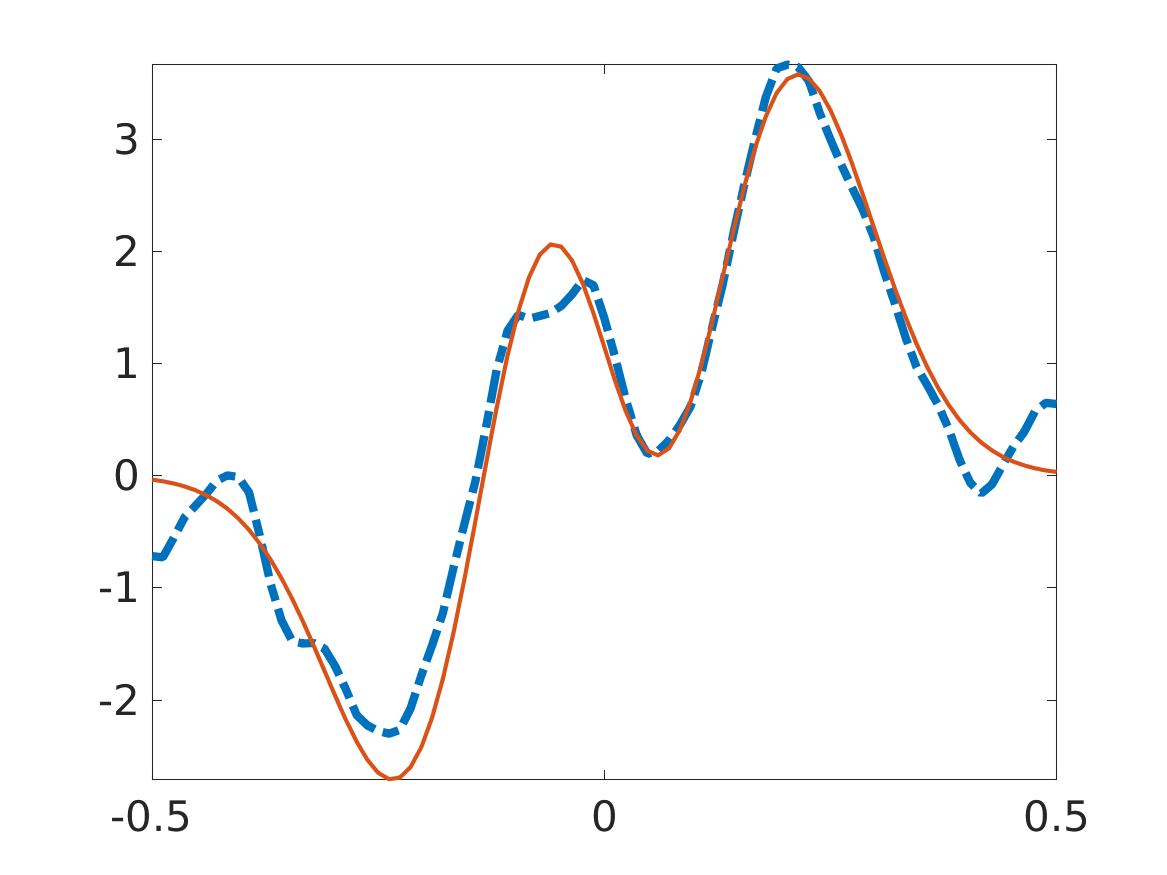}}  \quad
	\subfloat[The functions $p_{\rm comp}$ (dash-dot) and $p_{\rm true}$ (solid) on the line $y = 0$ when $\delta = 10\%$]{\includegraphics[width = 0.42\textwidth]{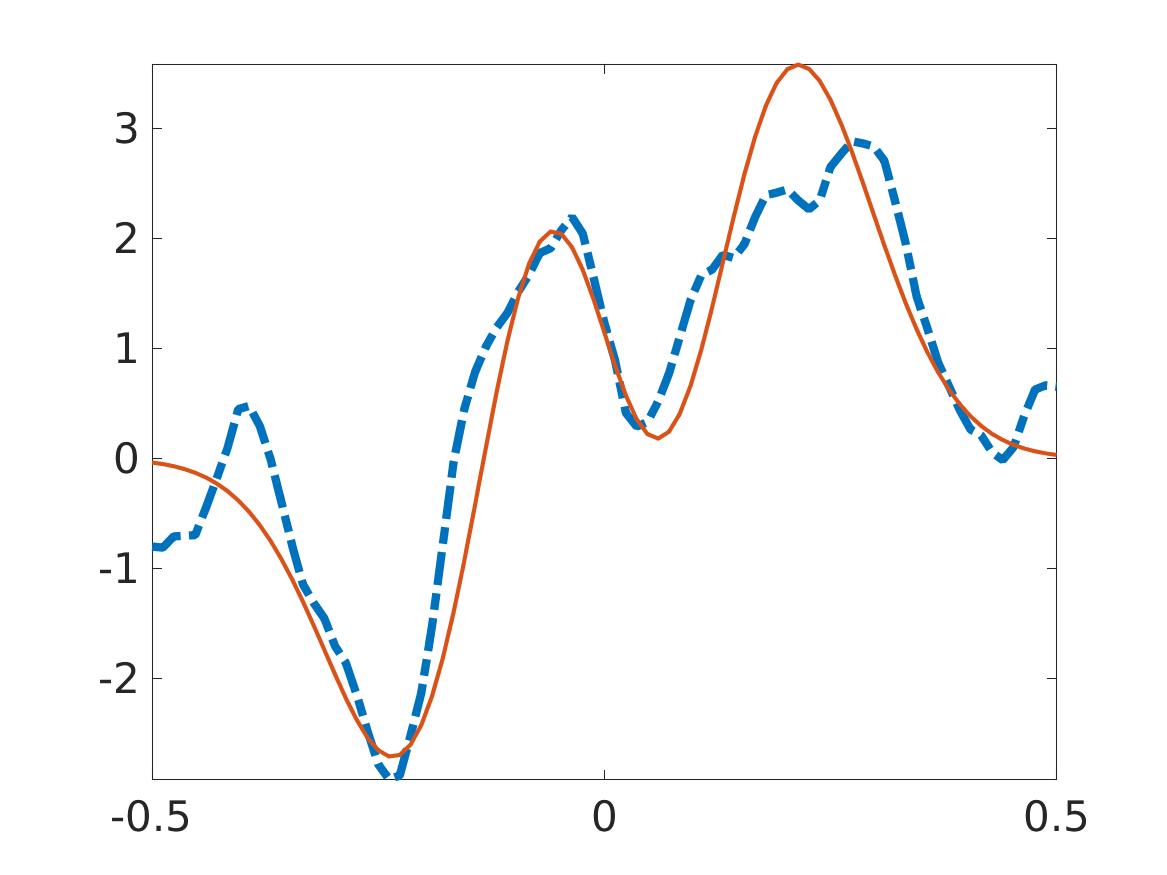}}
	\caption{\label{fig model 3} \it The true and reconstructed source functions in Test 3. The function $p^*$ has several local maxima and minima, which can be reconstructed efficiently.}
	\end{center}
\end{figure}

\noindent {\bf Test 4.} We try a non-smooth function $p^*(\x)$ as the ``negative" characteristic function of a letter ``A" and the ``positive" characteristic function of a letter $L$. It is interesting to see that our method recovers this step function very well, see Figure \ref{fig model 4}.

\begin{figure}[]
	\begin{center}
	\subfloat[The function $p_{\rm true}$]{\includegraphics[width = 0.42\textwidth]{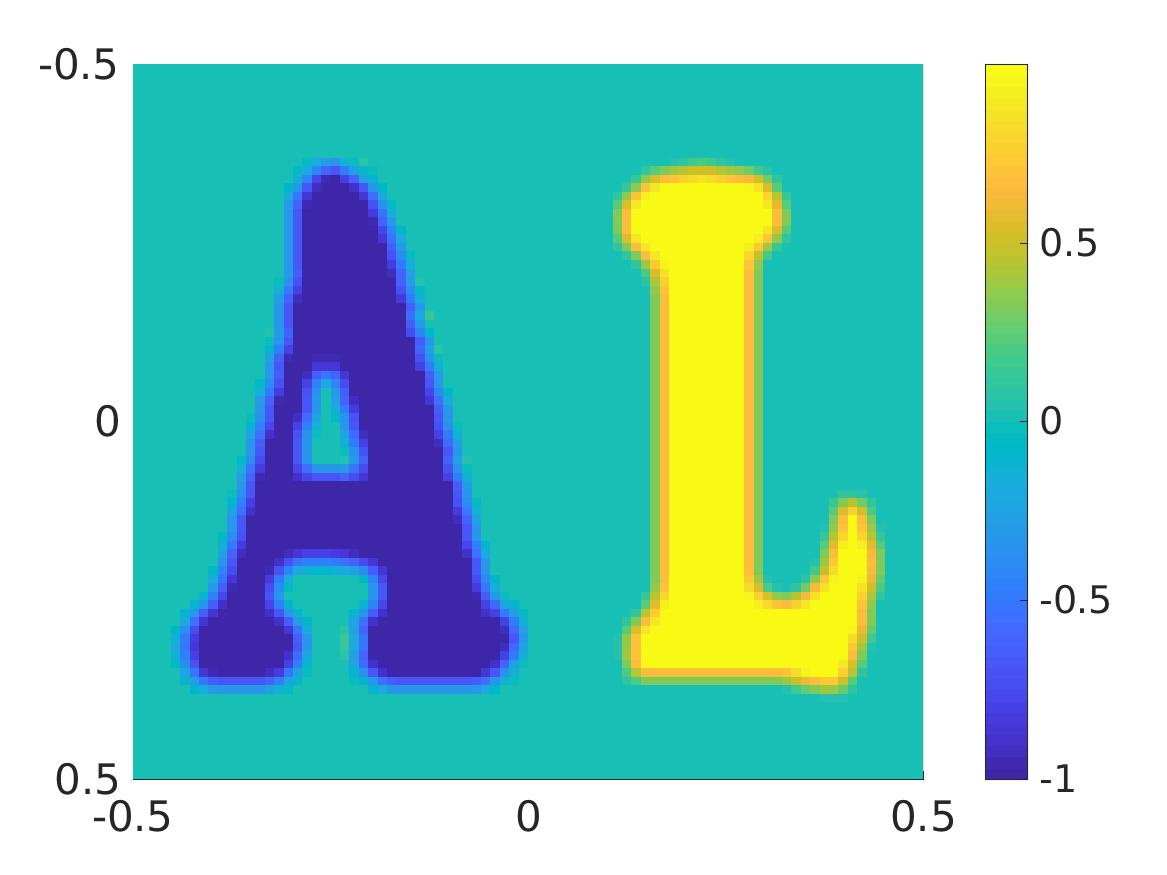}}  \quad
	\subfloat[The function $p_{\rm comp}$ when $\delta = 2\%.$]{\includegraphics[width = 0.42\textwidth]{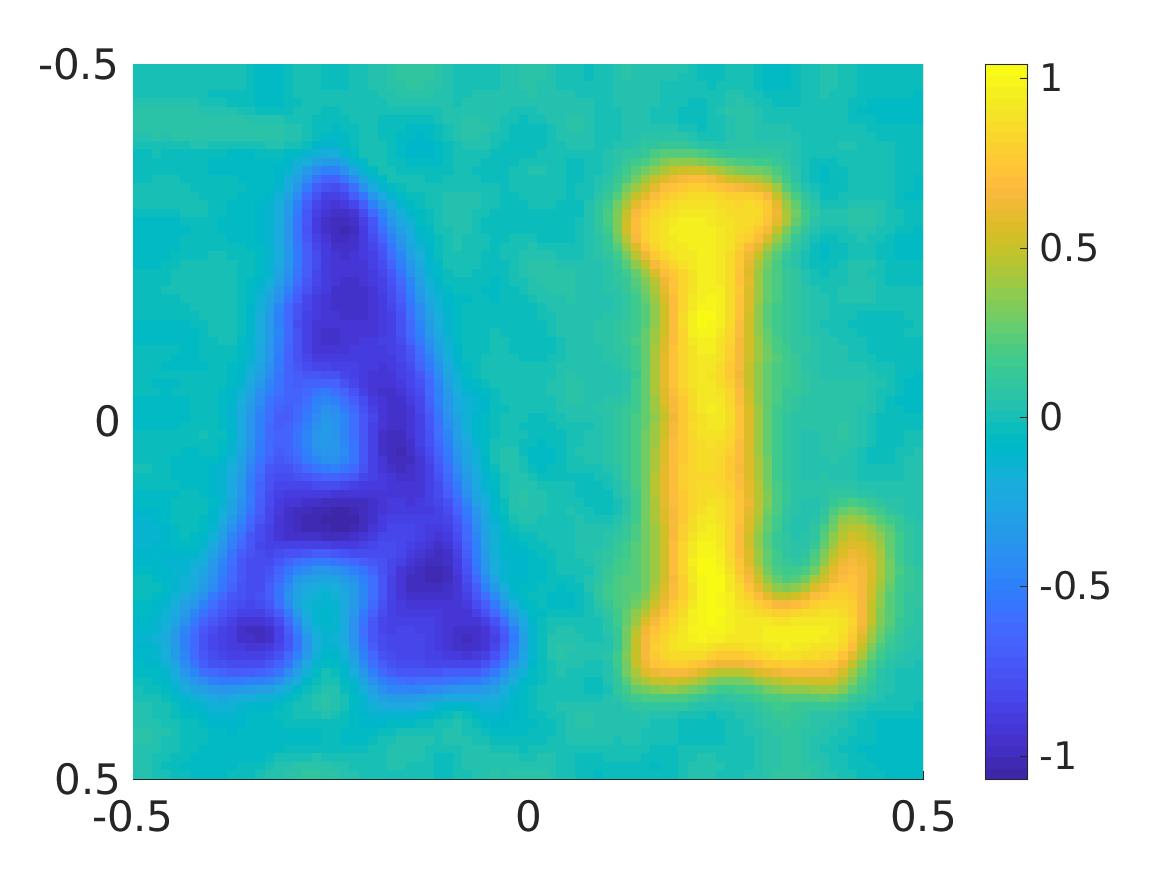}} 
	
		\subfloat[The function $p_{\rm comp}$ when $\delta = 5\%.$]{\includegraphics[width = 0.42\textwidth]{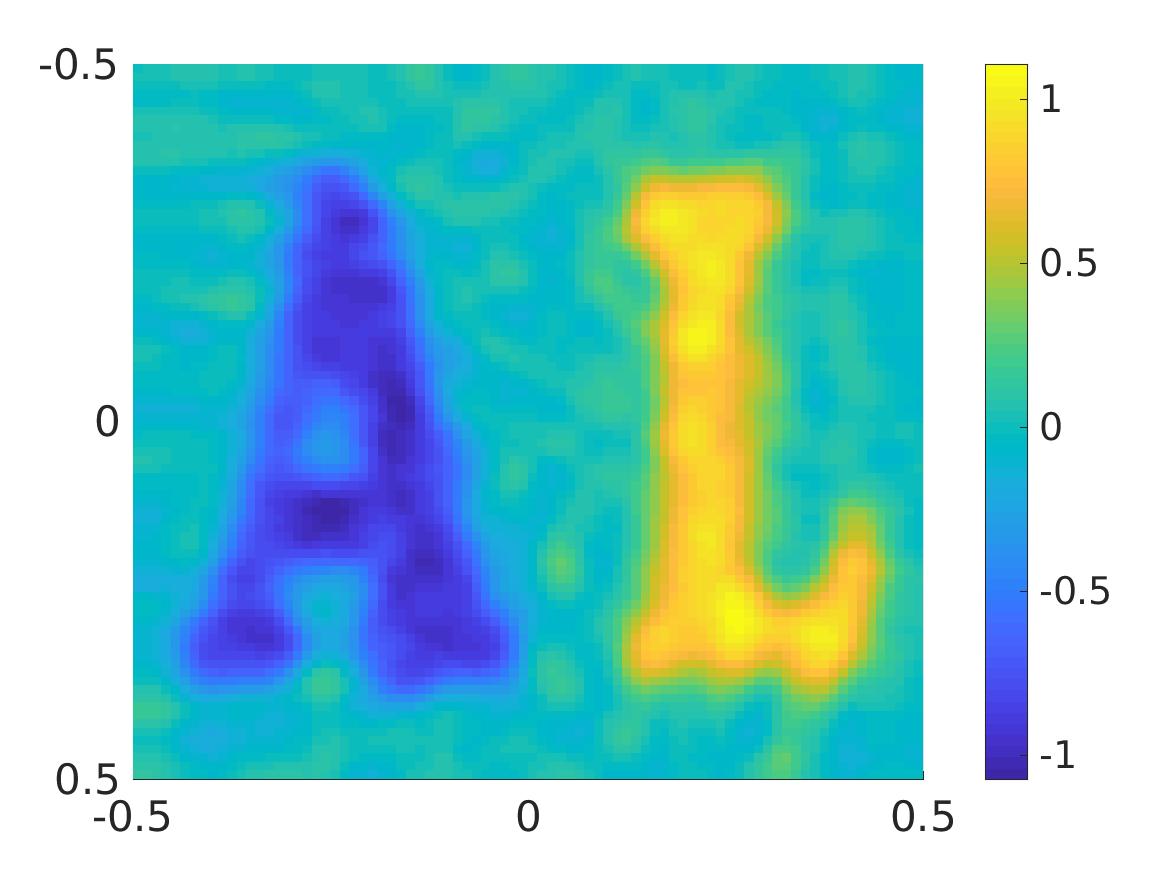}} \quad
	\subfloat[The function $p_{\rm comp}$ when $\delta = 10\%.$]{\includegraphics[width = 0.42\textwidth]{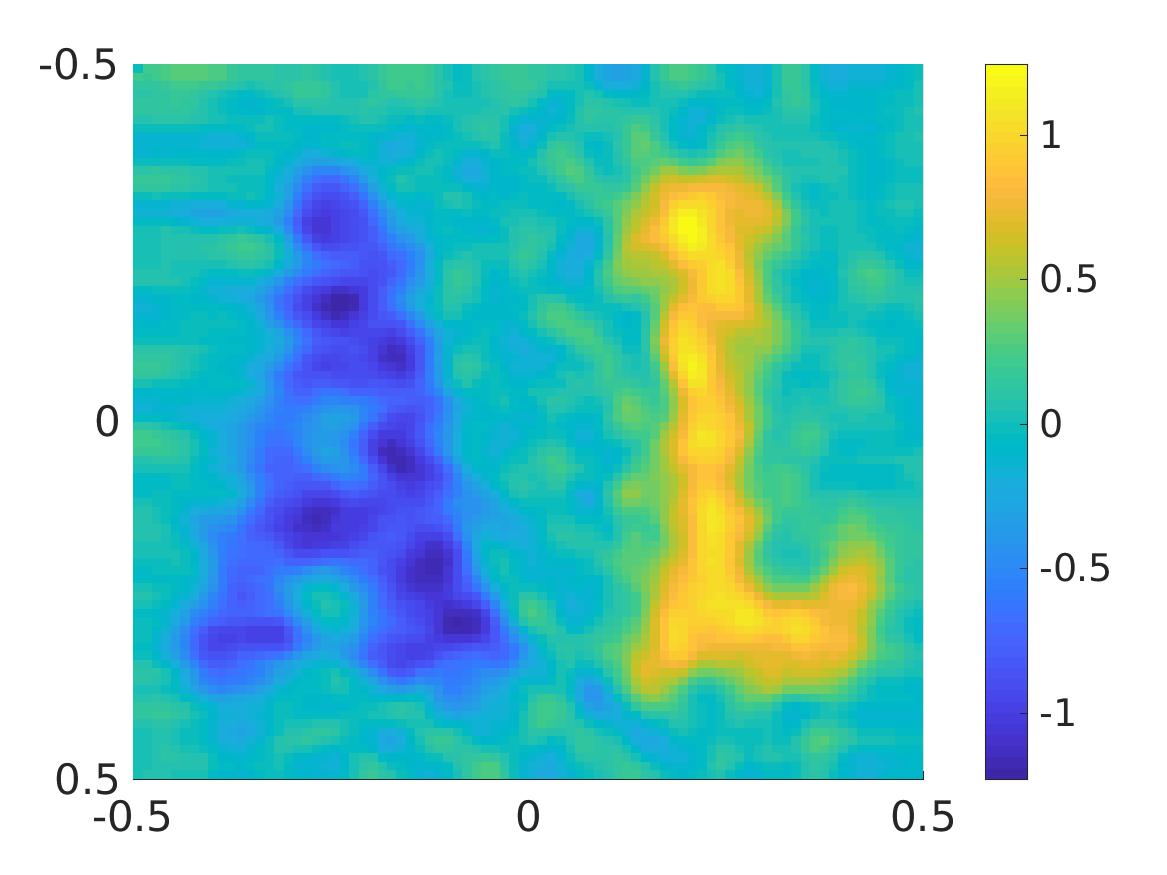}}
	
	\subfloat[The functions $p_{\rm comp}$ (dash-dot) and $p_{\rm true}$ (solid) on the line $y = 0$ when $\delta = 5\%$]{\includegraphics[width = 0.42\textwidth]{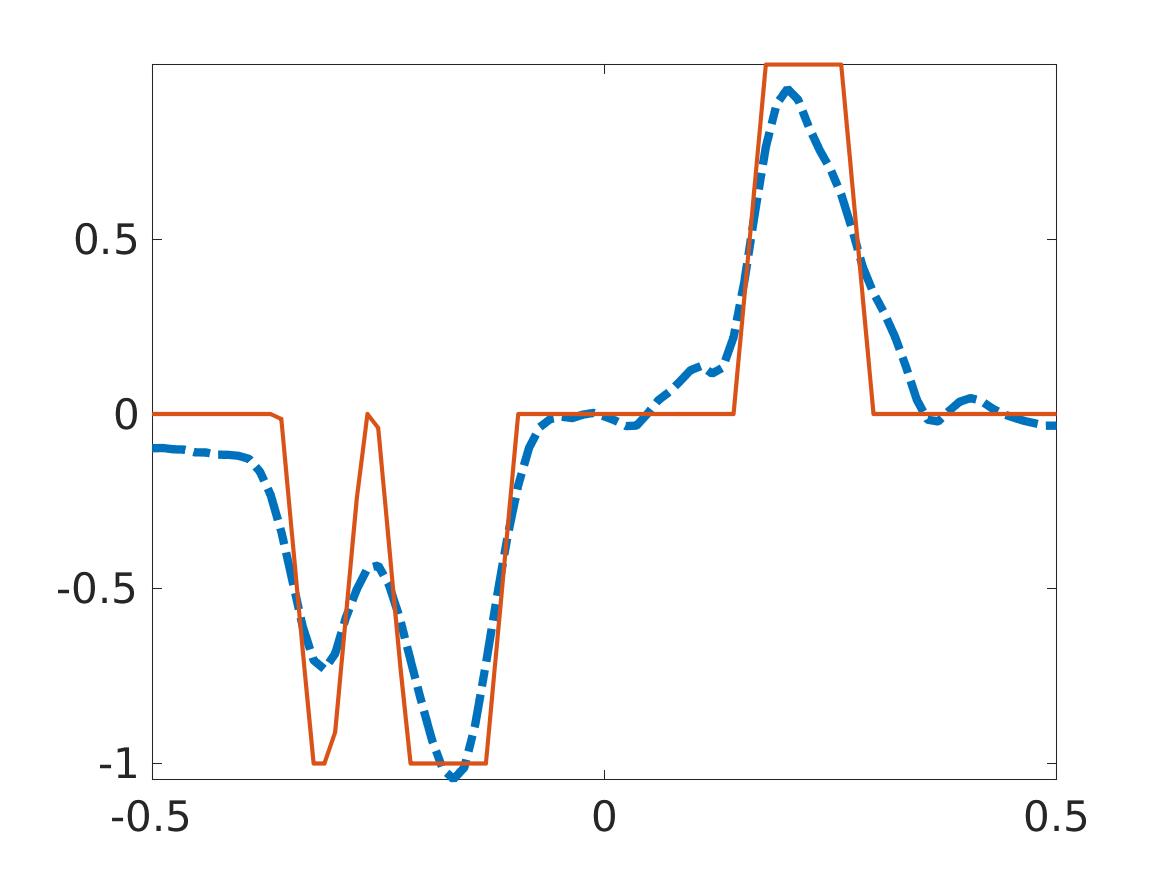}}  \quad
	\subfloat[The functions $p_{\rm comp}$ (dash-dot) and $p_{\rm true}$ (solid) on the line $y = 0$ when $\delta = 10\%$]{\includegraphics[width = 0.42\textwidth]{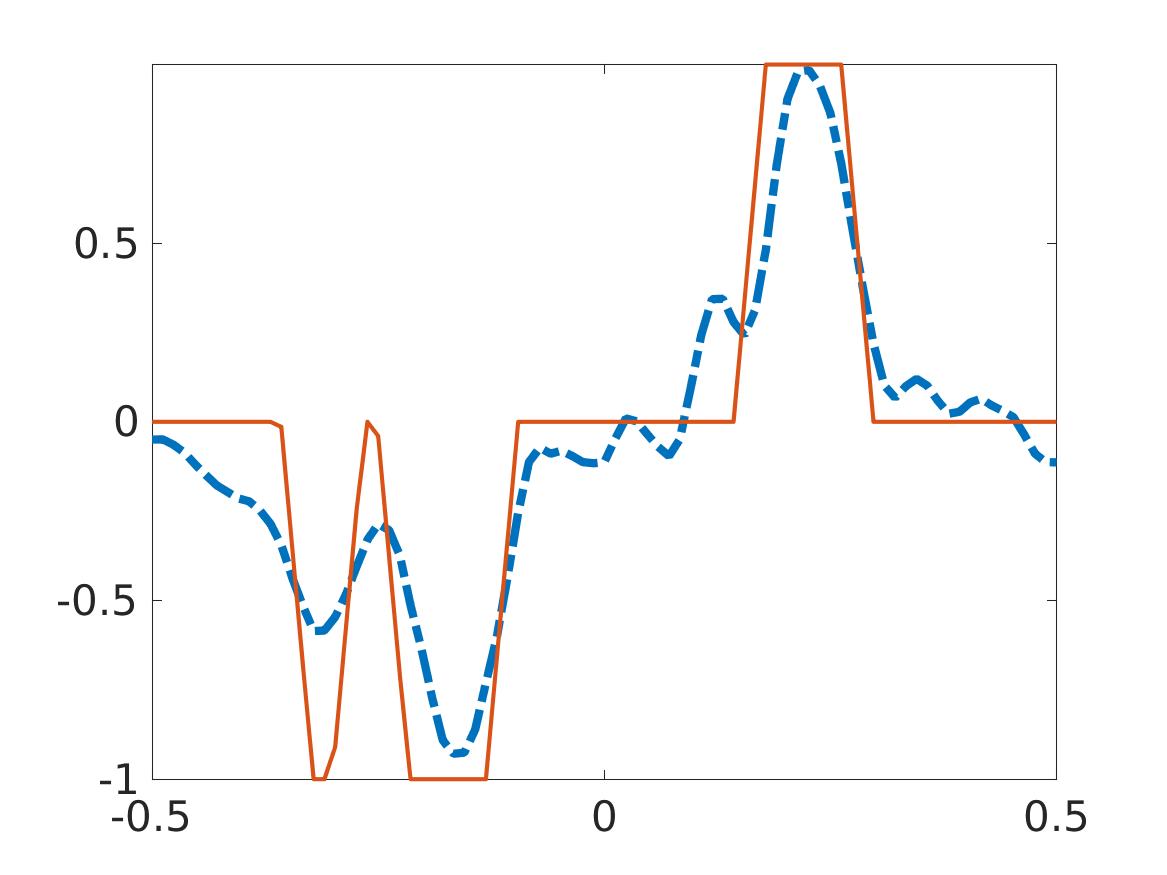}}
	\caption{\label{fig model 4} \it The true and reconstructed source functions in Test 4 where the data is noisy.  The reconstructions of nonconvex ``inclusions" $A$ and $L$ are satisfactory even when $\delta = 10\%$.}
	\end{center}
\end{figure}

\begin{remark}
We observe that the higher noise level, the blurrier reconstruction is. However, the reconstructed value of the function is stable. This can be seen by the color bar in Figures \ref{fig model 1}--\ref{fig model 4}. We list here the comparison of the minimum and maximum values of the true and reconstructed sources in Tables \ref{table 1} as a part of the numerical proof for the stability of the reconstruction method.
\end{remark}

\begin{table}[tbp]
\caption{\label{table 1} \it Correct and computed maximal and minimal values of source functions. $p_{\rm true}$ means the correct source functions. $p_{\rm comp}$ indicates the computed source functions. ${\rm error}_{\rm rel}$ denotes the relative error.}
\begin{center}
\begin{tabular}{|c|c|c|c|c|c|}
\hline
Test id& noise level & $\min p_{\rm true}$& $\min p_{\rm comp}$ (${\rm error}_{\rm rel}$) &  $\max p_{\rm true}$ & $\max p_{\rm comp}$ (${\rm error}_{\rm rel}$)\\ 
\hline
1 & $2\%$ & -2.0 &-1.99 (0.5\%) &2 & 2.00 (0.0\%)\\
\hline
2 & $ 2\%$&-1.0&-1.03 (3.0\%)&1.5&1.46 (2.7\%)\\
\hline
3& $ 2\%$& -6.5&-5.84 (10.1\%) & 8.1 & 7.36 (9.1\%)\\
\hline
4 & $ 2\%$&-1.0 & -1.07 (7.0\%)&1&1.04(4.0\%)\\
%
%
\hline
1 & $ 5\%$ & -2.0 &-2.14 (7.0\%) &2 & 2.14 (7.0\%)\\
\hline
2 & $  5\%$&-1.0&-1.03 (3.0\%)&1.5&1.46 (2.7\%)\\
\hline
 3& $5\%$& -6.5&-5.53 (15.0\%) & 8.1 & 7.36 (9.1\%)\\
\hline
 4 & $5\%$&-1.0 & -1.07 (7.0\%)&1&1.10(10.0\%)\\
\hline
%
%
%
1& $10\%$ & -2.0 &-2.33 (16.5\%) &2 & 2.24 (12.0\%)\\
\hline
 2& $10\%$&-1.0&-1.10 (10.0\%)&1.5&1.56 (4.0\%)\\
\hline
3& $10\%$& -6.5&-6.24 (4.0\%) & 8.1 & 7.30 (9.9\%)\\
\hline
4& $10\%$&-1.0 & -1.23 (23.0\%)&1&1.24(24.0\%)\\
\hline
\end{tabular}%
\end{center}
\end{table}

\section{Concluding remarks} \label{sec rem}

We have established in this paper a robust numerical method to solve an
inverse source problem for hyperbolic equations using the lateral Cauchy
data. Our method consists of deriving an integro-differential equation
involving a Volttera-like integral and then solving it by the well-known
quasi-reversibility method. We have established the Lipschitz-like
convergence rate of regularized solutions. This result is an extension of
the known convergence result for quasi-reversibility method for the case when a pure hyperbolic partial differential equation
is in place without a Volterra integral in it. A Carleman estimate is
essential in the analysis. Accurate numerical results are obtained.

Studying the ISP in the cases when the data is measured only on a part of
the boundary of the domain under consideration is reserved for a near future
research. Moreover, also in the future research, we will extend this method
to study ISPs for parabolic and elliptic equations.

As mentioned in Section \ref{sec problem statement}, the governing equation
for the ISP in this paper is the linearization of a coefficient inverse
problem, which is highly nonlinear. The numerical method developed in this
paper, therefore, can be used as a refinement step in solving that severely
ill-posed and highly nonlinear problem. For example, one might refine
numerical results obtained by the convexification globally convergent
numerical method, see, e.g. \cite{KlibanovKolesov:cma2018, KlibanovAlex:SIAMjam:2017}.

%
%

\begin{center}
\textbf{Acknowledgements}
\end{center}

This work was partially supported by research funds FRG 111172 provided by
University of North Carolina at Charlotte and by US Army Research Laboratory and Office of Army Research grant W911NF-15-1-0233. 
The author is grateful to Michael Klibanov for many fruitful  discussions.

\bibliographystyle{siam}
\bibliography{mybib}
 
\end{document}